\documentclass[11pt,twoside]{amsart}
\usepackage{amssymb}
\usepackage{graphicx,color}

\newtheorem{thm}{Theorem}[section]
\newtheorem{prop}[thm]{Proposition}

\newtheorem{lem}[thm]{Lemma}
\newtheorem{cor}[thm]{Corollary}
\newtheorem{defn}[thm]{Definition}

\newtheorem{ex}[thm]{Example}
\newtheorem{rem}[thm]{Remark}

\newcommand{\skipit}[1]{{}}
\newcommand{\prfend}{\hbox to7pt{\hfil}
\par\vskip-\baselineskip\hbox to\hsize
{\hfil\vbox {\hrule width6pt height6pt}}\vskip\baselineskip}

\newcommand{\ZZ}{\mathbb{Z}}
\newcommand{\RR}{\mathbb{R}}

\newcommand{\TC}{\mathbb{T}}
\newcommand {\PP}{\mathbb{P}}

\newcommand{\cO}{\mathcal{O}}

\DeclareMathOperator{\Image}{Im}

\DeclareMathOperator{\Proj}{Proj}

\newcommand{\myarrow}[2]{\hbox to #1pt{\hfil$\to$\hfil}{\hskip-#1pt{\raise
10pt\hbox to#1pt{\hfil$\scriptscriptstyle #2$\hfil}}}}

\begin{document}

\title{The minimal number of generators of a  Togliatti system}
\author[Emilia Mezzetti]{Emilia Mezzetti}
\address{Dipartimento di Matematica e  Geoscienze, Universit\`a di
Trieste, Via Valerio 12/1, 34127 Trieste, Italy}
\email{mezzette@units.it}

\author[Rosa M. Mir\'o-Roig]{Rosa M. Mir\'o-Roig}
\address{Facultat de Matem\`atiques, Department d'Algebra i
Geometria, Gran Via des les Corts Catalanes 585, 08007 Barcelona,
Spain}
\email{miro@ub.edu}

\begin{abstract} We compute the minimal and the maximal bound on the number of generators of a minimal smooth monomial Togliatti system of forms of degree $d$ in $n+1$ variables, for any $d\ge 2$ and $n\geq 2$.  We classify the Togliatti  systems with number of generators reaching the lower bound or close to the lower bound.  We then prove that if $n=2$ (resp $n=2,3$) all range between the lower and upper bound is covered, while if $n\geq 3$ (resp. $n\ge 4$) there are gaps if we only consider smooth minimal Togliatti systems (resp. if we avoid the smoothness hypothesis).  We finally analyze for $n=2$ the Mumford-Takemoto stability  of the syzygy bundle associated to smooth monomial Togliatti systems.
\end{abstract}

\thanks{Acknowledgments:  The first author is member of GNSAGA and is
supported by FRA, Fondi di Ricerca di Ateneo, and PRIN \lq\lq Geometria
delle variet\`a algebriche''.  The second  author was partially   supported
by  MTM2013-45075-P.
\\ {\it Key words and phrases.} Osculating space, weak Lefschetz
property, Laplace equations,
toric varieties.
\\ {\it 2010 Mathematic Subject Classification.} 13E10,  14M25, 14N05,
14N15, 53A20.}

\maketitle

\tableofcontents

\markboth{E. Mezzetti, R. M. Mir\'o-Roig}{The minimal number of
generators of a  Togliatti system}


\large

\section{Introduction}
The classification of the smooth projective varieties satisfying  at least one Laplace equation
is a classical problem, still very far  from being solved. We recall that a projective variety $X\subset \PP^N$ is said to satisfy a Laplace equation of order $d$, for an integer $d\geq 2$, if its $d$-osculating space at a general point has dimension strictly less than  expected.  The most famous example  is the Togliatti surface, a rational surface in $\PP^5$ parametrized by cubics, obtained from the $3$rd Veronese embedding $V(2,3)$ of $\PP^2$ by a suitable projection from four  points:  the Del Pezzo surface obtained  projecting $V(2,3)$ from three general points on it admits  a  point which belongs to all its osculating spaces, so  projecting further from this special point one obtains a surface having all osculating spaces of dimension $\leq 4$ instead of the expected  $5$. This surface is named from Eugenio Togliatti who gave a classification of rational surfaces parametrized by cubics and satisfying at least one Laplace equation of order $2$. For more details see the original articles of Togliatti \cite{T1}, \cite{T2}, or   \cite{LM},  \cite{V}, \cite{I} for discussions of this example.  In \cite{MMO} the two authors of this note and Ottaviani described a connection, due to apolarity, between projective varieties satisfying at least one Laplace equation and homogeneous artinian ideals in a polynomial ring, generated by polynomials of the same degree and failing the weak Lefschetz property (WLP for short). Let us recall that a homogeneous ideal $I\subset R:=K[x_0, \ldots, x_n]$ fails the weak Lefschetz property in some degree $j$ if, for any linear form $L$, the map of multiplication by $L$ from $(R/I)_j$ to $(R/I)_{j+1}$ is not of maximal rank (see \cite{MM}). Thanks to this connection, explained in detail in Section 2, they obtained in the toric case the classification of the smooth rational threefolds parametrized by cubics and satisfying a Laplace equation of order 2, and gave a conjecture to extend it to varieties of any dimension. This conjecture has been recently proved in \cite{MRM}. Note that the assumption that the variety is toric translates in the fact that the related ideals are generated by monomials, which simplifies apolarity and allows to exploit combinatorial methods. This point of view had been introduced by Perkinson  in \cite{P}, and applied  to the classification of toric surfaces and threefolds satisfying Laplace equations under some rather strong additional assumptions on the osculating spaces.

In this note we begin the study of the analogous problems for smooth toric rational varieties parametrized by monomials of degree $d\geq4$, or equivalently for artinian  ideals of $R$ generated by monomials of degree $d$.
The picture becomes soon much more involved than in the case of cubics, and for the moment a complete classification appears out of reach. We consider mainly minimal smooth toric Togliatti systems  of forms of degree $d$ in $R$, i.e. homogeneous artinian ideals generated by monomials failing the WLP, minimal with respect to this property, and such that the apolar linear system parametrizes a smooth variety.

The first goal of this note is to establish minimal and maximal bounds, depending on $n$ and $d\ge 2$, for the number of generators of  Togliatti systems of this form,  and to classify the  systems reaching the minimal bound, or close to reach it. We then investigate if all values comprised between the minimal and the maximal bound can be obtained as number of generators of a minimal smooth Togliatti system. We prove that the answer is positive if $n=2$, but negative if $n\geq 3$.  If we avoid smoothness assumption, the answer becomes positive for $n=3$ but is still negative for $n\geq 4$, even we detect some intervals and sporadic values that are reached. Finally, as applications of our results, we study  the Mumford--Takemoto stability of the syzygy bundle associated to a minimal smooth Togliatti system with $n=2$.

\vskip 2mm
Next we outline the structure of this note.  In
Section~\ref{defs and  prelim results} we  fix the notation
and we collect the basic results on Laplace equations and the
Weak Lefschetz Property needed in the sequel.
Section~\ref{minimalnumbergenerators} contains the main results of this note. Precisely, after recalling the results for degree $2$ and $3$, in Theorem \ref{mainthm1} we prove that  the  minimal bound $\mu^s(n,d)$ on the number of generators of a minimal smooth Togliatti system of forms of degree $d$ in $n+1$ variables, for $d\geq 4$, is equal to $2n+1$, and classify  the systems reaching the bound. Then in Theorem \ref{mainthm2} we get the complete classification for systems with number of generators $\mu^s(n,d)+1$. We also compute the maximal bound $\rho^s(n,d)$ and give various examples.
In Section~\ref{number} we prove that for $n=2$ and any $d\geq 4$ all numbers in the range between $\mu^s(n,d)$ and $\rho^s(n,d)$ are reached  (Proposition \ref{interval}), while for $n\geq 3$ the value $2n+3$ is a gap (Proposition \ref{2n+3} ). We then prove that, avoiding smoothness, for $n=3$ the whole interval is covered.
Finally Section~\ref{associatedbundles} contains the results about stability of the syzygy bundle for minimal smooth monomial Togliatti systems in $3$ variables.

\vskip 2mm\noindent {\bf Notation.}
Throughout this work  $k$ will be an algebraically closed field of
characteristic zero
and $\PP^n=\Proj(k[x_0,x_1,\cdots ,x_n])$.
We denote by $V(n,d)$ the Veronese variety
image of the projective space $\PP^n$ via the $d$-tuple
Veronese embedding.
$(F_1,\ldots,F_r)$ stands for the ideal generated by $F_1,\ldots,F_r$,
while $\langle F_1,\ldots,F_r \rangle$
denotes the $k$-vector space they generate.

\vskip 2mm \noindent {\bf Acknowledgement.} Part of this work was done
while the second author was a guest of the University
of Trieste and she thanks the University of Trieste for its hospitality.
The authors wish to thank the referee for some useful remarks.

\section{Background and preparatory results} \label{defs and prelim results}

In this section, we recall some standard
terminology and notation
from commutative algebra and algebraic geometry,
as well as some results needed later on. In particular, we briefly
recall  the relationship between   the
existence of homogeneous  artinian ideals $I\subset
k[x_0,x_1,\dots,x_n]$ which fail the weak Lefschetz property; and  the
existence of (smooth) projective varieties $X\subset \PP^N$ satisfying
at least one Laplace equation of order $s\ge 2$.
For more details,  see \cite{MMO} and \cite{MRM}.

\vskip 2mm
\noindent
{\bf A. The Weak Lefschetz Property.}
Let $R : = k[x_0,x_1,\dots,x_n]=\oplus _tR_t$ be the graded polynomial
ring in $n+1$ variables over  the
 field $k$.

\begin{defn}\label{def of wlp}\rm Let $I\subset R$ be a
homogeneous artinian ideal. We  say that   $R/I$
 has the {\em weak Lefschetz property} (WLP, for short)
if there is a linear form $L \in (R/I)_1$ such that, for all
integers $j$, the multiplication map
\[
\times L: (R/I)_{j} \to (R/I)_{j+1}
\]
has maximal rank, i.e.\ it is injective or surjective.
We will often abuse notation and say that the ideal $I$ has the
WLP.   In this
case, the linear form $L$ is called a {\em Lefschetz element} of
$R/I$.  If for the general form $L \in (R/I)_1$
and for an integer number $j$
the map $\times L$ has not maximal rank we will say that
the ideal $I$ fails the WLP in degree $j$.
\end{defn}

The Lefschetz elements of $R/I$ form a Zariski
open, possibly empty, subset of $(R/I)_1$.
  Part of the great
interest in the WLP stems from the  ubiquity of its
presence (See, e.g., \cite{BK}, \cite{CN1}, \cite{HMMNWW}, \cite{HMNW},
\cite{LZ} - \cite{MR}) and the fact that its presence puts
severe constraints on the possible Hilbert functions, which can
appear in various disguises (see, e.g.,
\cite{St-faces}).
  Though many algebras
 are expected to have the
WLP, establishing this property is often rather difficult. For
example, it was shown by R. Stanley \cite{st} and J. Watanabe
\cite{w} that a monomial artinian complete intersection ideal
$I\subset R$ has the WLP. By semicontinuity, it follows that
a {\em general}   artinian complete intersection ideal $I\subset
R$ has the WLP but it is open whether {\em every} artinian complete
intersection of height $\ge 4$
over a field of characteristic zero has the WLP. It is worthwhile to
point out that
the weak Lefschetz
property of an artinian ideal $I$ strongly depends on the characteristic
of the ground field $k$ and,
 in positive characteristic, there are examples of artinian complete
intersection ideals $I\subset k[x_0,x_1,x_2]$ failing the WLP (see,
e.g., Remark 7.10 in \cite{MMN2}).

In \cite{MMO}, Mezzetti, Mir\'{o}-Roig, and Ottaviani showed  that  the
failure of the
WLP can be used to construct  (smooth) varieties satisfying at least
one   Laplace equation of order $s\ge 2$ (see also \cite{MRM} and \cite{A}).
Let us review the needed concepts from differential geometry in order to
state this result.

\vskip 2mm \noindent

{\bf B. Laplace Equations.}

Let $X\subset \PP^N$ be a projective variety of dimension $n$ and
let $x\in X$ be a smooth point. We choose a system of
affine coordinates and an analytic local parametrization $\phi $ around $x$ where
 $x=\phi(0,...,0)$ and the $N$
components of $\phi$ are formal power series.
The {\em $s$-th osculating  space}
$T_x^{(s)}X$ to $X$ at $x$ is the projectivised span of  all
partial derivatives of $\phi$ of order
$\leq s$.
The expected dimension of $T_x^{(s)}X$
is ${{n+s}\choose {s}}-1$, but
in general $\dim T_x^{(s)}X\leq {{n+s}\choose {s}}-1$;
if strict inequality
holds for all smooth points of $X$, and $\dim
T_x^{(s)}X={{n+s}\choose {s}}-1-\delta$ for a general point $x$,
then $X$ is said to satisfy $\delta$ Laplace equations of order $s$.

\begin{rem}\label{bound} \rm It is clear that if $N<{{n+s}\choose
{s}}-1$ then $X$
satisfies at least one Laplace equation
of order $s$, but this case is not interesting and
will not be considered in the following.
\end{rem}

\vskip 2mm

Let $I$ be an artinian
ideal
generated by $r$
homogeneous polynomials
  $F_1, \cdots ,F_r\in R$ of  degree $d$. Associated to $I_d$ there is a
morphism 
$$\varphi _{I_d}:\PP^n \longrightarrow \PP^{r-1}.$$
Note that $\varphi _{I_d}$ is everywhere regular because $I$ is an artinian ideal.
Its image $X_{n,I_d}:={\Image (\varphi _{I_d})}\subset \PP^{r-1}$
is the projection of the $n$-dimensional Veronese variety $V(n,d)$
from the linear system $\langle(I^{-1})_d \rangle\subset \mid \cO
_{\PP^n}(d)\mid=R_d$ where
 $I^{-1}$ is the ideal generated by the Macaulay inverse system of $I$
(See \cite{MMO}, \S 3 for details). Analogously, associated to
  $(I^{-1})_d$ there is a rational map
$$\varphi _{(I^{-1})_d}:\PP^n \dashrightarrow \PP^{{n+d
\choose d}-r-1}.$$ The closure of its image
$X_{n,(I^{-1})_d}:=\overline{\Image (
\varphi _{(I^{-1})_d})}\subset \PP^{{n+d\choose d}-r-1}$
is the projection of the $n$-dimensional Veronese variety
$V(n,d)$ from the linear system $\langle F_1,\cdots ,F_r \rangle \subset
\mid \cO _{\PP^n}(d)\mid=R_d$.
The varieties  $X_{n,I_d}$ and
$X_{n,(I^{-1})_d}$ are usually called  apolar.
In the following $X_{n,(I^{-1})_d}$ will simply be denoted by $X$.
\vskip 2mm

 We have:

\begin{thm}\label{teathm} Let $I\subset R$ be an artinian
ideal
generated
by $r$ homogeneous polynomials $F_1,...,F_{r}$ of degree $d$.
If
$r\le {n+d-1\choose n-1}$, then
  the following conditions are equivalent:
\begin{itemize}
\item[(1)] the ideal $I$ fails the WLP in degree $d-1$;
\item[(2)] the  homogeneous forms $F_1,...,F_{r}$ become
$k$-linearly dependent on a general hyperplane $H$ of $\PP^n$;
\item[(3)] the $n$-dimensional   variety
 $X=X_{n,(I^{-1})_d}$ satisfies at least one Laplace equation of order
$d-1$.
\end{itemize}
\end{thm}

\begin{proof} See \cite[Theorem 3.2]{MMO}.
\end{proof}

In view of Remark \ref{bound}, the assumption
$r\le {n+d-1\choose
n-1}$ ensures that the Laplace equations obtained in (3)
are not obvious. In the particular
case $n=2$, this assumption
gives  $r\leq d+1$.

The above result motivates the following definition:

\begin{defn} \label{togliattisystem} \rm Let $I\subset R$ be an artinian
ideal
generated
by $r$ forms $F_1,...,F_{r}$ of degree $d$, $r\le {n+d-1\choose n-1}$.
We introduce the following definitions:
\begin{itemize}
\item[(i)] $I$  is a \emph{Togliatti system}
if it satisfies the three equivalent conditions in  Theorem \ref{teathm}.

\item[(ii)]   $I$ is a \emph{monomial Togliatti system} if,
in addition, $I$ (and hence $I^{-1}$) can be generated by monomials.

\item[(iii)]   $I$ is a \emph{smooth Togliatti system} if,
in addition, the $n$-dimensional   variety
$X$ is smooth.

\item[(iv)] A  monomial  Togliatti system $I$ is said to be
\emph{minimal} if $I$ is generated by monomials $m_1, \cdots ,m_r$ and
there is no proper subset $m_{i_1}, \cdots ,m_{i_{r-1}}$ defining a
monomial  Togliatti system.
    \end{itemize}
\end{defn}

The names are in honor of Eugenio Togliatti who proved that for $n=2$ the only
smooth  Togliatti system of cubics is
$I=(x_0^3,x_1^3,x_2^3,x_0x_1x_2)\subset k[x_0,x_1,x_2]$ (see \cite{BK}, \cite{T1},
\cite{T2}).
 The main goal of our note is to determine a lower bound $\mu (n,d)$
(resp. $\mu ^s(n,d)$) for the minimal number of generators $\mu(I)$ of
any  (resp. smooth) minimal monomial Togliatti system  $I\subset
k[x_0,x_1,\cdots ,x_n]$ of forms of degree $d\ge 2$ and classify {\em
all} (resp.  smooth) minimal monomial Togliatti systems $I\subset
k[x_0,x_1,\cdots ,x_n]$ of forms of degree $d\ge 2$ which reach the
bound, i.e. $\mu(I)=\mu(n,d)$ (resp. $\mu(I)=\mu ^s(n,d)$). These
results will be achieved in the next section.



\section{The minimal number of generators of a smooth Togliatti system}
\label{minimalnumbergenerators}

From now on, we restrict our attention to  monomial artinian ideals
 $I\subset k[x_0,\ldots ,x_n]$ (i.e. the ideals invariants for the
natural toric action of $(k^*)^n$). Recall that  when $I\subset R$ is an
artinian monomial ideal, the  homogeneous part  $I^{-1}_d$ of degree $d$
of the inverse system  $I^{-1}$
is spanned by the monomials in $R_d$  not in $I$. It is also worthwhile
to recall that for monomial artinian ideals
to test the WLP there
is no need to consider a general linear form. In fact, we have

\begin{prop}
   \label{lem-L-element}
Let $I \subset R:=k[x_0,x_1,\cdots , x_n]$ be an artinian monomial ideal.
Then $R/I$ has the WLP if and only if  $x_0+x_1 + \cdots + x_n$
is a Lefschetz element for $R/I$.
\end{prop}

\begin{proof} See \cite{MMN2};  Proposition 2.2.
\end{proof}

Given an artinian ideal $I\subset k[x_0,x_1,\cdots ,x_n]$, we denote by
$\mu (I)$ the minimal number of generators of $I$. We define
$$\mu(n,d):=min \{ \mu(I) \mid I\in \mathcal{ T}(n,d) \}, $$
$$\mu^s(n,d):=min \{ \mu(I) \mid I\in \mathcal{ T}^s(n,d) \}, $$
$$\rho (n,d):=max \{ \mu(I) \mid I\in \mathcal{ T}(n,d) \} \text{ and }$$
$$\rho ^s(n,d):=max \{ \mu(I) \mid I\in \mathcal{ T}^s(n,d) \} $$
where $\mathcal{ T}(n,d)$ is the set of all minimal  monomial  Togliatti
systems  $I\subset k[x_0,x_1,\cdots ,x_n]$ of forms of degree $d$ and
$\mathcal{ T}^s(n,d)$ is the set of all minimal smooth monomial
Togliatti systems  $I\subset k[x_0,x_1,\cdots ,x_n]$ of forms of degree
$d$. By definition, we have $\mathcal{ T}^s(n,d) \subset \mathcal{ T}(n,d)$.

Our first goal is to provide  a lower bound for $\mu(n,d)$ and $\mu ^s(n,d)$.
 First, we observe that all artinian monomial ideals $I\subset
k[x_0,x_1,\cdots ,x_n]$
 generated by  forms of degree $d\ge 2$ contain $x_i^d$ for $i=0, \cdots
,n$ and the ideals $(x_0^d,\cdots ,x_n^d)$ do satisfy WLP. Therefore, we
always have
 \begin{equation} \label{ineq} n+2\le \mu(n,d) \le \mu^s(n,d)\le  \rho
^s(n,d)\le \rho (n,d)\le {n+d-1\choose n-1}.\end{equation}

Let us start analyzing the cases $d=2,3$.
\vskip 2mm
\begin{rem}\rm
The minimal smooth monomial  Togliatti systems  $I\subset
k[x_0,x_1,\cdots ,x_n]$ of quadrics were classified in \cite{MRM};
Proposition 2.8.  It holds:
\begin{itemize}
\item[(i)] $\mathcal{ T}^s(2,2)=\emptyset $.
\item[(ii)] For $n\ge 3$, we have $$\mu ^s(n,2)=\begin{cases}
\lambda^2+2\lambda +1 & \text{ if } n=2\lambda \\\lambda ^2+3\lambda +2
& \text{ if } n=2\lambda +1.
\end{cases}$$
\item[(iii)] For $n\ge 3$, $\rho ^s(n,2)={n\choose 2}+3$.
\end{itemize}
In particular, for $n=3$ we have $n+2<\mu ^s(n,2)=\rho
^s(n,2)={n+1\choose 2}$; for $n=4$ we have  $n+2<\mu ^s(n,2)=\rho
^s(n,2)<{n+1\choose 2}$; and for all $n>4$ the inequalities in
(\ref{ineq}) are strict, i.e.,
$$n+2 <\mu^s(n,2)< \rho^s(n,2)< {n+1\choose 2}.$$
 We also have  $\mu(n,2)=2n+1$ for $n\ge 4$ (since we easily check that
$\mu(n,2)\ge 2n+1$ and $I=(x_0^2,x_1^2,\cdots
,x_n^2,x_0x_1,x_0x_2,\cdots ,x_0x_n)$ fails weak Lefschetz property from
degree 1 to degree 2) and  $\mu(3,2)=6$ (since $\mu (3,2)>5$ and
$I=(x_0^2,x_1^2,x_2^2,x_3^2,x_0x_1,x_2x_3)$ fails weak Lefschetz
property in degree 1).
\end{rem}

\vskip 2mm
\begin{rem} \rm
The minimal smooth monomial  Togliatti systems  $I\subset
k[x_0,x_1,\cdots ,x_n]$ of cubics  were classified in \cite{MMO}; Theorem 4.11 and \cite{MRM};
Theorem 3.4. It holds:
\begin{itemize}
\item[(i)] $\rho ^s(2,3)=\mu ^s(2,3)=4$,
\item[(ii)] $\rho ^s (3,3)=\mu ^s(3,3)=8$,
\item[(iii)] $13= \mu^s(4,3)<15=\rho^s(4,3)$, and
\item[(iv)] For all $n\ge 4$, we have $\rho^s(n,3)={n+1\choose
3}+n+1$,$$\mu ^s(n,3)=min\{ \sum_{i=1}^s{a_i+2\choose 3}+\sum _{1\le
i<j<k\le s}a_ia_ja_k \mid n+1=\sum _{i=1}^s a_i \text{ and } n-1\ge
a_1\ge \cdots \ge a_s\ge 1 \}$$
    $$=\begin{cases} 2{\lambda +3\choose 3}  & \text{ if } n=2\lambda +1 \\
     {\lambda +2\choose 3}  +2{\lambda +3\choose 3}  & \text{ if }
n=2\lambda
    \end{cases}$$
    and, hence $$n+2 <\mu^s(n,3)< \rho^s(n,3)< {n+2\choose 3}.$$
     \end{itemize}

\vskip 2mm
 We may also check that $\mu(n,3)=2n+1$ for $n\ge 3$ (since $\mu(n,3)\ge
2n+1$ and $I=(x_0^3,x_1^3,\cdots ,x_n^3,x_0^2x_1,x_0^2x_2,\cdots
,x_0^2x_n)$ fails weak Lefschetz property in degree 2)
and  $\mu(2,3)=4$ (since $\mu (2,3)\ge 4$ and
$I=(x_0^3,x_1^3,x_2^3,x_0x_1x_2)$ fails weak Lefschetz property from
degree 2 to degree 3).
Notice that $\mu ^s(n,2)\ge 2n+1$ unless $n=2, 3$ and $\mu ^s(n,3)\ge
2n+1$ unless $n=2,3$.
\end{rem}

\vskip 2mm
From now on, we assume $d\ge 4$ and  $n\ge 2$. We will prove that
$\mu ^s (n,d)= \mu  (n,d)= 2n+1$. In addition,  we will  classify all (resp.
smooth)  minimal monomial Togliatti systems $I\subset k[x_0,x_1,\cdots
,x_n]$ of forms of degree $d\ge 4$ with $\mu(I)=2n+1$  and  all smooth
minimal monomial Togliatti systems $I\subset k[x_0,x_1,\cdots ,x_n]$ of
forms of degree $d\ge 4$  with $\mu(I)=\mu^s(n,d)+1=2n+2$,  revealing how
the power of combinatorics tools can allow us to deduce pure geometric
properties of projections of $n$-dimensional Veronese varieties $V(n,d)$.
To prove it, we will associate to any artinian monomial ideal a polytope
and the toric variety $X=X_{n,(I^{-1})_d}$ introduced in \S 2 B. Hence,
we will be able to tackle our problem with tools  coming from
combinatorics. In fact, when we deal with artinian monomial ideals
$I\subset k[x_0,x_1,\cdots,x_n]$  the failure of the WLP can be
established by fairly easy combinatoric properties of the associated
polytope $P_I$. To state this result we need to  fix some extra notation.

\vskip 2mm
Let $I\subset k[x_0,x_1,\cdots,x_n]$ be an artinian monomial ideal
generated by monomials  of degree $d$ and let $I^{-1}$ be its inverse
system.  We denote by $\Delta _n$ the standard $n$-dimensional simplex
in the lattice $\ZZ^{n+1}$, we consider $d\Delta _n$ and we define the
polytope $P_I$ as the convex hull of the finite subset $A_I\subset
\ZZ^{n+1}$ corresponding to monomials of degree $d$  in $I^{-1}$. As
usual we define the sublattice $\text{ Aff}_{\ZZ}(A_I)$ in $\ZZ^{n+1}$
generated by $A_I$ as follows:
$$\text{ Aff}_{\ZZ}(A_I):=\{ \sum _{x\in A_I} n_x\cdot x \mid n_x\in
\ZZ, \quad \sum _{x\in A_I} n_x=1 \}.$$

\vskip 2mm

We have:

\begin{prop} \label{failureWLP}
Let $I\subset k[x_0,x_1,\cdots,x_n]$ be an artinian monomial ideal
generated by $r$ monomials of degree $d$. Assume $r\le {n+d-1\choose n-1}$.
Then, $I$ is a Togliatti system if and only if there exists a
hypersurface of degree $d-1$ containing  $A_I\subset \ZZ^{n+1}$.
 In addition, $I$ is a minimal Togliatti system
  if and only if any such hypersurface $F$  does not contain any
integral point of $d\Delta_n\setminus A_I$ except possibly  some of the
vertices of $d\Delta_n$.
\end{prop}
\begin{proof} It follows from   Theorem \ref{teathm} and \cite{P},
Proposition 1.1.
\end{proof}

Let us illustrate the above proposition with a precise example.

\begin{ex} \rm
The artinian ideal $I=(x_0,x_1)^3+(x_2,x_3)^3\subset
k[x_0,x_1,x_2,x_3]$  defines a minimal monomial Togliatti system of
cubics. In fact, the set  $A_I\subset \ZZ^4$ is:
$$A_I=\{ (2,0,1,0), (1,0,2,0),(2,0,0,1),(1,0,0,2),(0,2,1,0),(0,1,2,0),$$
$$(0,2,0,1),(0,1,0,2),
(1,1,1,0),(1,1,0,1),(1,0,1,1),(0,1,1,1)\}.$$
There is a hyperquadric,  and only one, containing all points of
$A_I$ and no integral point of $3\Delta_3\setminus A_I$, namely,
$$Q(x_0,x_1,x_2,x_3)=2(x_0^2+x_1^2+x_2^2+x_3^2)+4(x_0x_1+x_2x_3)-5(x_0x_2+x_0x_3+x_1x_2+x_1x_3).$$

\end{ex}

For seek of completeness we also recall the following useful
combinatorial criterion which will allow us to check if a subset $A$ of
points in the lattice $\ZZ^{n+1}$ defines a smooth toric variety $X_A$
or not.

\begin{prop} \label{smoothness}
 Let $I\subset k[x_0,x_1,\cdots,x_n]$ be an artinian monomial ideal
generated by monomials  of degree $d$. Let $A_I\subset \ZZ^{n+1}$ be the
set of integral points corresponding to monomials in $(I^{-1})_d$, $S_I$
the semigroup generated by $A_I$ and 0, $P_I$ the convex hull of $A_I$
and $X_{A_I}$ the projective toric variety associated to the polytope $P_I$.
 $X_{A_I}$ is smooth if and only if for any non-empty face $\Gamma $ of
$P_I$ the following conditions hold:

 (i) The semigroup $S_I/\Gamma$ is isomorphic to $\ZZ^m_+$ with
$m=\dim(P_I)-\dim \Gamma +1$.

 (ii) The lattices $\ZZ^{n+1}\cap \text{ Aff}_{\RR}(\Gamma )$ and
$\text{ Aff}_{\ZZ}(A_I\cap \Gamma)$ coincide.

\end{prop}
\begin{proof} See \cite{GKZ}; Chapter 5, Corollary 3.2. Note that in
this case  $X_{A_I}=X_{n,(I^{-1})_d}$.
\end{proof}

\vskip 2mm
Figure 1 illustrates two examples of minimal Togliatti systems which are
non-smooth. The points of the complementary of $A_I$ are marked with a cross.

\vskip 4mm

\begin{figure}[h]\label{1}
    {\includegraphics[width=40mm]{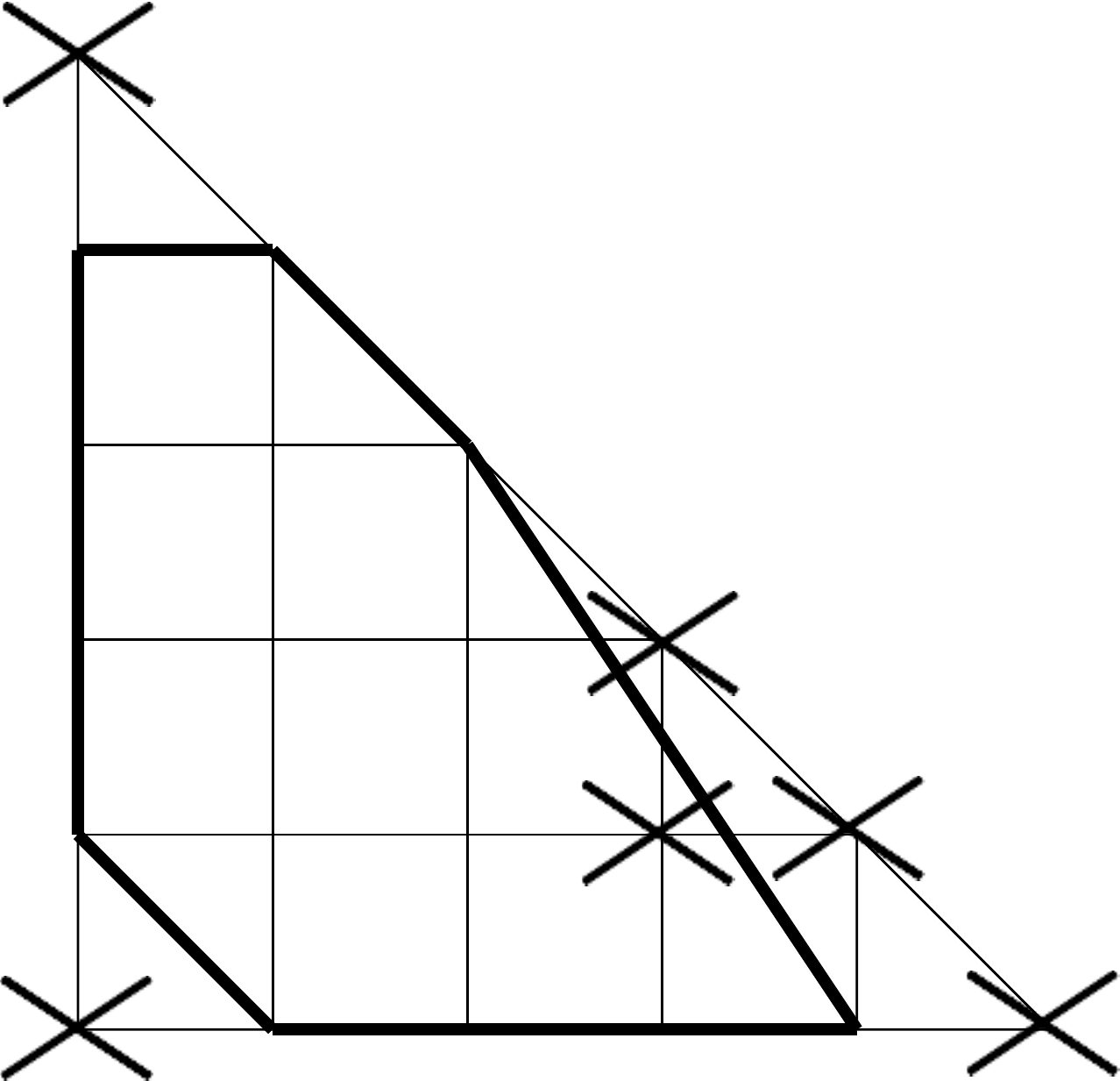}\  \  \  \  \ \  \  \
 \  \  \  \  \includegraphics[width=40mm]{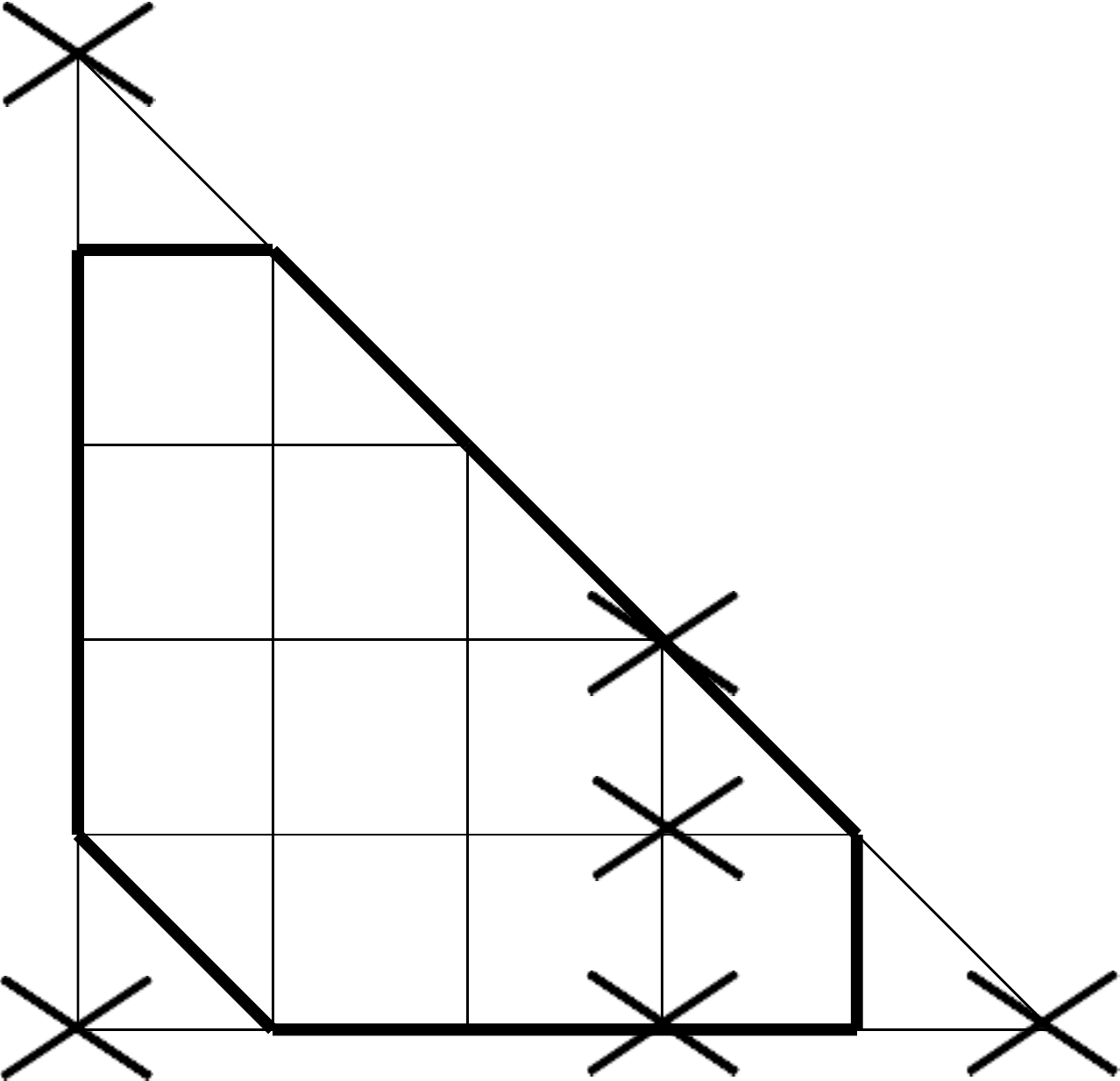}}
    \caption{Non-smooth Togliatti systems with $n=2$ and $d=5$}
    \end{figure}

 The condition (i) of Proposition \ref{smoothness} is verified if and only if  translating each vertex $v$ of the polygon to the origin of $\ZZ^2$ and considering for each edge coming out of $v$  the first point with integer coordinates, these form a $\ZZ$-basis of $\ZZ^2$. The condition (ii) is equivalent to require that each point of $\ZZ^2$ which lies on an edge of the polygon is also a point of $A_I$. Therefore, the first figure violates condition (i) and the second one violates condition (ii).

\vskip 4mm
In order to achieve the classification of minimal (resp. smooth)
monomial  Togliatti systems $I\subset k[x_0,\cdots ,x_n]$ of degree
$d\ge 4$ with $\mu(I)$ as small as possible we need to introduce  one
more definition.

\begin{defn} \rm A Togliatti system  $I\subset k[x_0,x_1,\cdots,x_n]$ of
forms of degree $d$ is said to be {\em trivial} if
there exists a form $F$ of degree $d-1$  such that $I$ contains $x_0F,
\cdots,x_nF$.
\end{defn}

The following remark justifies why we call them trivial.
\begin{rem} \rm
(i) Let  $F$ be a homogeneous form of degree $d-1$. Since
$x_0F, x_1F, \cdots ,x_nF$
become linearly dependent  on the hyperplane $x_0+\cdots +x_n=0$,  using
Proposition  \ref{lem-L-element}, we conclude that any artinian  ideal
of the form
$I=( x_0,\cdots ,x_n)F+( F_1,\cdots
,F_s)$
is a (trivial) Togliatti system. In the monomial case, looking at the
inverse system that parameterizes the surface $X$, we can observe
that it satisfies a Laplace equation of the simplest form, given by the
annihilation of the partial derivative of order $d-1$ corresponding to
the monomial $F$.

(ii) Let $I\subset k[x_0,x_1,\cdots ,x_n]$ be a monomial Togliatti
system of cubics. If $I$ is trivial  then it is not smooth. \end{rem}

\begin{thm}\label{mainthm1} For any integer $n\ge 2$ and $d\ge 4$, we
have $\mu ^s(n,d)=\mu(n,d)=2n+1$. In particular, if
$I\subset k[x_0,x_1, \cdots ,x_n]$ is a minimal (resp.  smooth minimal)
 monomial Togliatti system of forms of degree $d$, then $\mu(I)\ge 2n+1$.

In addition, {\em all}  minimal  monomial Togliatti systems $I\subset
k[x_0,\cdots ,x_n]$ of forms of degree $d\ge 4$ with $\mu(I)=2n+1$ are
trivial  
unless  one of the following cases holds:
 \begin{itemize}
 \item[(i)] $(n,d)=(2,5)$ and, up to a permutation of the coordinates,
$I=(x_0^5,x_1^5,x_2^5,x_0^3x_1x_2,x_0x_1^2x_2^2)$.
 \item[(ii)] $(n,d)=(2,4)$ and, up to a permutation of the coordinates,
$I=(x_0^4,x_1^4,x_2^4,x_0x_1x_2^2,x_0^2x_1^2)$.
 \end{itemize}

 Furthermore, (i) is smooth and (ii) is not smooth.
\end{thm}
\begin{proof} First of all we observe that  $I=(x_0^d,x_1^d,\cdots
,x_n^d)+x_0^{d-1}(x_1,\cdots , x_n)\subset k[x_0,\cdots ,x_n]$ is a
minimal monomial Togliatti system of forms of degree $d$ and by
Proposition \ref{smoothness}, being $d\geq 4$, it is smooth.  Thus,
$\mu(n,d)\le \mu^s(n,d)\le 2n+1$ .

To prove that $\mu(n,d)=2n+1$, we have to check that any monomial
artinian ideal  $I=(x_0^d,\cdots ,x_n^d,x_0^{a^1_0}x_1^{a^1_1}\cdots
x_n^{a_n^1},\cdots ,x_0^{a_0^{n-1}}x_1^{a^{n-1}_1}\cdots
x_n^{a_n^{n-1}})$ with $\sum _{i=0}^na_{i}^j=d\ge 4$, $1\le j\le n-1$
has the WLP at the degree  $d-1$.
According to Proposition \ref{failureWLP}, to prove the last assertion
it is enough to prove that  no  hypersurface of degree $d-1$ contains
all  points of $A_I\subset \ZZ^{n+1}$, where as before  $A_I\subset
\ZZ^{n+1}$ is the set of all integral  points corresponding to monomials
of degree $d$ in $I^{-1}$.
For any integer $0\le i \le d$, we set $H_i=\{(a_0,\cdots ,a_n)\in
\ZZ^{n+1} \mid a_0=i\}$ and   $A_I^i:=A_I\cap H_i$; we have $A_I=\cup
_{i=0}^dA_I^d$.

\vskip 2mm
To illustrate this method, in Figure 2   we show the pictures of the sets $A_I$, and  $A_I^0$, $A_I^1$, $A_I^2$, $A_I^3$, 
when $I=(x_0^4, x_1^4, x_2^4, x_0^2x_1x_2)$.
\begin{figure}[h]\label{2}
    {\includegraphics[width=40mm]{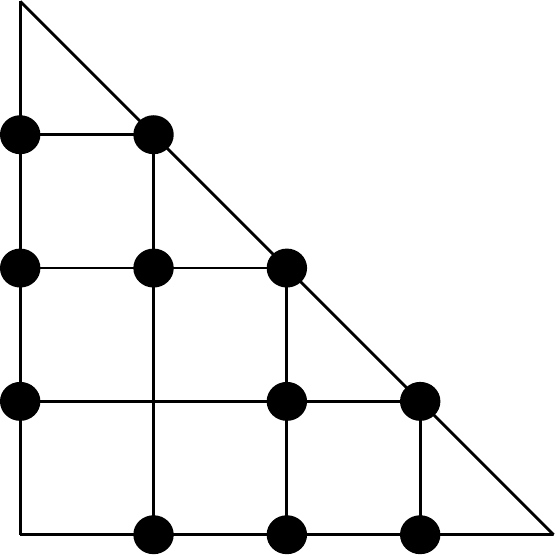}}
    \caption{$A^I$ with $I=(x_0^4,x_1^4, x_2^4, x_0^2x_1x_2)$}
    \end{figure}
\begin{figure}[h]
    {\includegraphics[width=35mm]{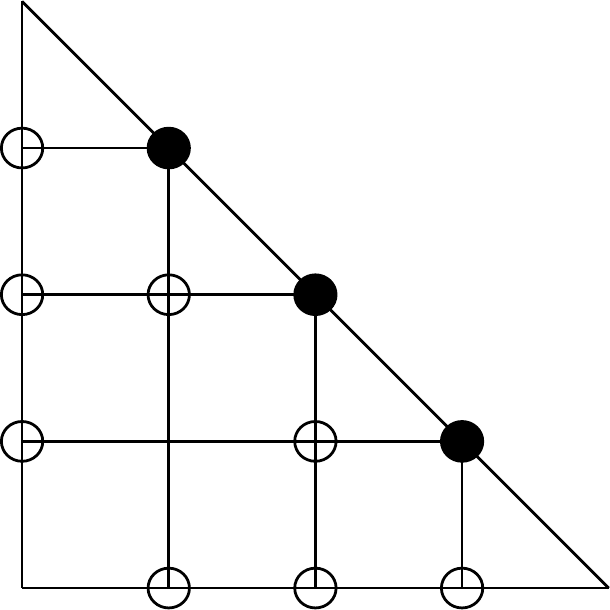}
 \  \   \includegraphics[width=35mm]{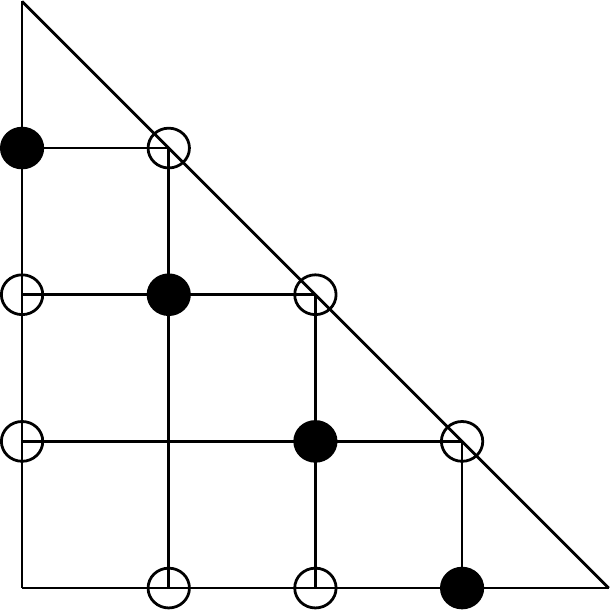}
 \   \  \includegraphics[width=35mm]{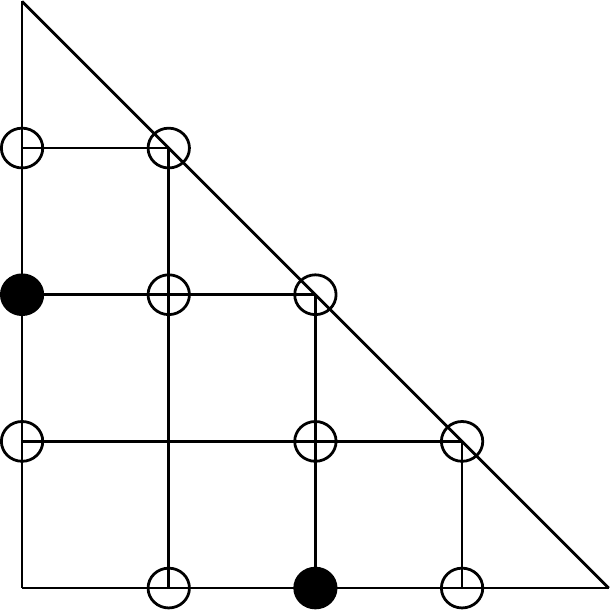} 
 \   \  \includegraphics[width=35mm]{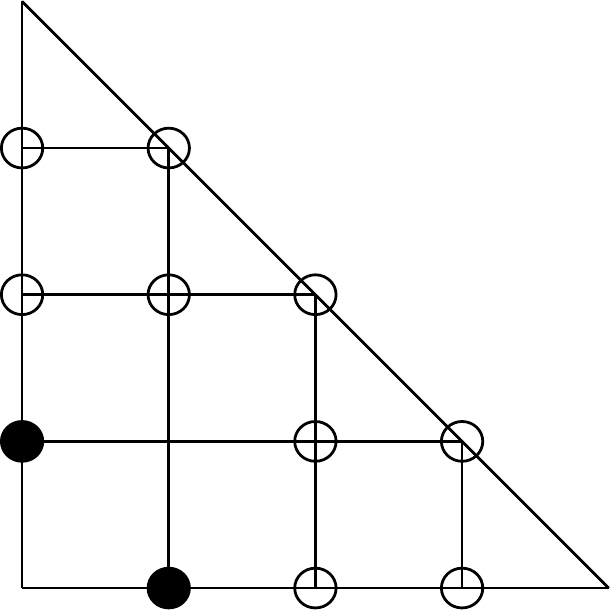} 
 
 $A_I^0$ \ \ \  \ \ \  \ \   \ \ \ \ \ \ \ \ \ \ \ \ \ \ \ \ \ \ \ $A_I^1$  \ \ \ \ \ \ \  \ \ \ \ \ \  \ \ \ \ \ \ \ \ \ \ \ \ \ \ $A_I^2$  \ \ \ \   \ \ \ \ \ \ \ \ \ \ \ \ \ \ \ \ \ \ \ \ \  $A_I^3$ }
    \end{figure}
\vskip 2mm

 We will prove now the theorem proceeding by induction on $n$. Let us start with the case
  $n=2$. We take  a monomial artinian ideal
$I=(x_0^d,x_1^d,x_2^d,x_0^{a_0^1}x_1^{a_1^1}x_2^{a_2^1})$ with
$a_0^1+a_1^1+a_2^1=d\ge 4$ and we show that  no  plane curve of degree
$d-1$ contains all points of $A_I\subset \ZZ^3$.  Since $4\le
d=a_0^1+a_1^1+a_2^1$, we can assume wlog that $2\le a_0^1$. We assume
that there is a plane curve $F_{d-1}$ of degree $d-1$ containing all
points of $A_I$ and we will get a contradiction. Since $F_{d-1}$
contains the $d$ points of $A_I^1$, it factorizes as
$F_{d-1}=L_1F_{d-2}$. Since $F_{d-2}$  contains the $d-1$ points
$A_I^0$, it factorizes as $F_{d-1}=L_0L_1F_{d-3}$. Now, if $a_0^1=2$,
then $A^2_I$  contains $d-2$ points, if $a^1_0>2$, then $A^2_I$
contains $d-1$ points; in any case $F_{d-3}=L_2F_{d-4}$ for a suitable
form $F_{d-4}$ of degree $d-4$. Repeating the argument we get that
$F_{d-1}=L_0L_1 \cdots L_{d-2}$, so $F_{d-1}$ does not contain the
points of $A_I^{d-1}$, which is non-empty by assumption. This
contradicts the existence
 of  a plane curve of degree $d-1$ containing all integral points of $A_I$.

Let now $n\geq 3$ and assume that the claim is true for $n-1$. Let us
prove that  no hypersurface  of degree $d-1$ contains all points of
$A_I\subset \ZZ^{n+1}$,
 where $$I=(x_0^d,\cdots ,x_n^d, x_0^{a^1_0}x_1^{a^1_1}\cdots
x_n^{a_n^1},\cdots ,x_0^{a_0^{n-1}}x_1^{a^{n-1}_1}\cdots
x_n^{a_n^{n-1}})$$ with $\sum _{i=0}^na_{i}^j=d\ge 4$, $1\le j\le n-1$.
Wlog we can assume $a_0^1\ge a_1^1\ge \cdots \ge a_n^1\ge 0$ and also
$a^1_0\geq a^2_0$. Therefore $a_0^1>0$, so $x_0$ appears explicitly in
the monomial $x_0^{a^1_0}x_1^{a^1_1}\cdots x_n^{a_n^1}$ and $A^0_I$ is
equal to  $d\Delta_{n-1}$ minus the $n$ vertices and at most $n-2$ other
points. By inductive assumption, no hypersurface in $n$ variables
  of degree $d-1$ contains $A^0_I$, so $F_{d-1}$
factorizes as $L_0F_{d-2}$, where $F_{d-2}$ is a hypersurface of degree
$d-2$ containing all points of $A_I\setminus A^0_I$.

 If  the $n-1$ monomials have $a_0^1=a_0^2=\cdots=a_0^{n-1}\le 1$, then
$A^2_I=(d-2)\Delta_{n-1}, \cdots, A^{d-1}_I=\Delta_{n-1}$ and we deduce
that $F_{d-1}=L_0L_2\cdots L_{d-1}$, because for $j=2,\cdots, d-1$ the
simplex $(d-j)\Delta_{n-1}$ is not contained in any hypersurface in
$n-1$ variables of degree $d-j$. This gives a contradiction because
$F_{d-1}$ misses all points of $A^1_I\neq\emptyset$. Otherwise,
$A^1_I=(d-1)\Delta_{n-1}$ minus at most $n-2$ points. Then by inductive
assumption there is no hypersurface of degree $d-1$ in $n-1$ variables
containing $A^1_I$. Then we repeat the argument until we reach a
contradiction.

   \vskip 2mm  Finally we will classify all minimal  monomial Togliatti
systems $I\subset k[x_0,\cdots ,x_n]$ of forms of degree $d\ge 4$ with
$\mu(I)=2n+1$. First we assume that $n=2$ and we will show that all of
them are  trivial  unless $d=5$ and
$I=(x_0^5,x_1^5,x_2^5,x_0^3x_1x_2,x_0x_1^2x_2^2)$ or $d=4$ and
$I=(x_0^4,x_1^4,x_2^4,x_0x_1x_2^2,x_0^2x_1^2)$.
    Take $I=(x_0^d,x_1^d,x_2^d, m_1, m_2)\subset k[x_0,x_1,x_n] $ with
$m_i=x_0^{a_0^i}x_1^{a_1^i} x_2^{a_2^i}$ and $\sum _{j=0}^2a_j^i =d$
a minimal Togliatti system. If there exists $0\le i \le 2$ such that
$a_i^1, a_i^2\ge 2$ (wlog we assume $i=0$) then the plane curve
$F_{d-1}$ containing all integral points of $A_I$ factorizes
$F_{d-1}=L_0L_1\cdots L_{d-2}$ and since $F_{d-1}$ cannot miss any point
of $A_I$  we must have $A_I^{d-1}=\emptyset $ which  forces
$m_1=x_0^{d-1}x_1$, $m_2=x_0^{d-1}x_2$. Assume now that for any $0\le i
\le 2$, there exists $1\le j \le 2$ with $a_i^j\le 1$. Since $d\ge 4$ we may assume
$a_0^1,a_1^1 \le 1$ and $a_2^2\le 1$. Therefore, $m_1\in \{ x_0x_1x_2^
 {d-2},x_0x_2^{d-1},x_1x_3^{d-1}\}$ and $m_2\in \{x_0^ax_1^{d-1-a}x_2,
x_0^{\alpha }x_1^{d-\alpha} \mid 0\le a, \alpha \le d-1 \}$.  But none
gives a minimal  Togliatti system because $ x_0^d,x_1^d,x_2^d,m_1,m_2$ are linearly independent
on a general line of $\PP^2$ (see Theorem \ref{teathm}) unless $d=5$ and $m_1=x_0x_1x_2^3$ and
$m_2=x_0^2x_1^2x_2$ or $d=4$ and $m_1=x_0^2x_1x_2$ and
$m_2=x_0x_1^2x_2^2$.
 Furthermore, applying Proposition \ref{smoothness},
we easily check that only
$I=(x_0^5,x_1^5,x_2^5,x_0^3x_1x_2,x_0x_1^2x_2^2)$ defines a smooth variety.

   Assume now $n\ge 3$ and $ d\ge 4$ and let $I=(x_0^d,x_1^d,\cdots
,x_n^d, m_1,\cdots, m_n)\subset k[x_0,\cdots ,x_n] $ with
$m_i=x_0^{a_0^i}x_1^{a_1^i}\cdots x_n^{a_n^i}$ and $\sum _{j=0}^na_j^i
=d$ be a   Togliatti system. There is an integer $j$, $0\le j\le n$ such
that  $\# \{i\mid a_j^i\ge 1\}\ge 2$. Therefore, wlog we can assume
$a_0^1,a_0^2\ge 1$. Arguing as in the previous part of the proof any
hypersurface $F_{d-1}$ of degree $d-1$ containing all integral points of
$A_I$ factorizes $F_{d-1}=L_0L_1\cdots L_{d-2}$ and since $F_{d-1}$
cannot miss any point of $A_I$  we must have $A_I^{d-1}=\emptyset $
which  forces $m_1=x_0^{d-1}x_1$, $m_2=x_0^{d-1}x_2,\cdots ,
m_n=x_0^{d-1}x_{n}$ and hence $I$ is trivial,  which proves what we want.
\end{proof}

\begin{rem} \rm  Minimal monomial Togliatti systems $I\subset
k[x_0,x_1,x_2]$ of forms of degree $d\ge 4$ with $\mu(I)=5$ were also
classified by Albini in \cite{A}; Theorem 3.5.1. So, our results can be
seen as a generalization of his result to the case of an arbitrary
number of variables.
\end{rem}

\begin{rem}\label{triv} \rm
Up to permutation of the variables, the trivial Togliatti systems with
$\mu(I)=2n+1$ are of the form $(x_1^d, \cdots, x_n^d)+x_0^{d-1}(x_0,
\cdots, x_n)$.
\end{rem}

Figure \ref{3} illustrates the only smooth non--trivial example of minimal
Togliatti system of forms of degree $5$ with $\mu(I)=5$.
\vskip 4mm

\begin{figure}[h]\label{3}
    {\includegraphics[width=40mm]{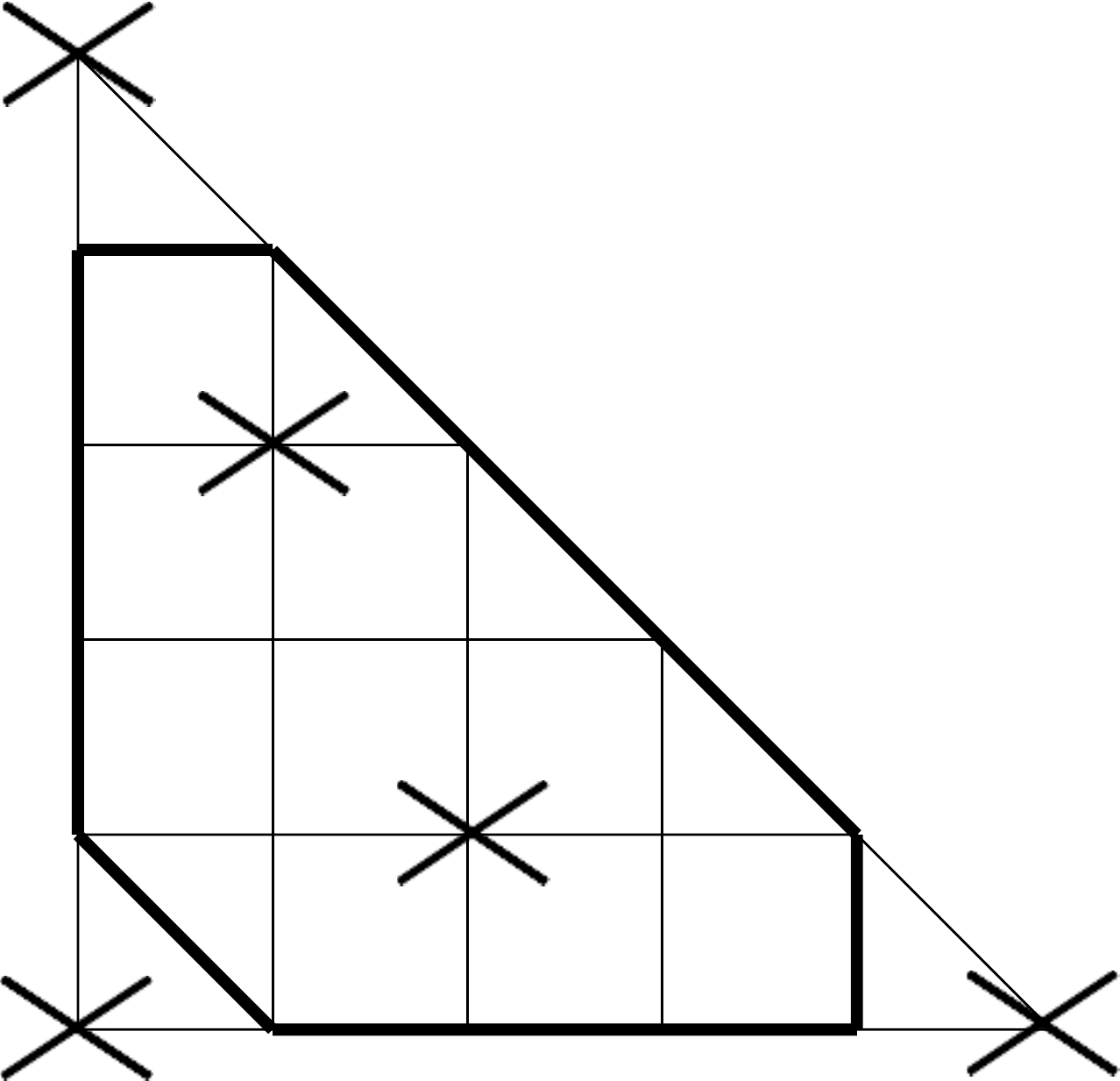}}
    \caption{Smooth non--trivial Togliatti system with $n=2$ and $d=5$}
    \end{figure}

\begin{rem} \rm
In the case of the non-trivial minimal smooth monomial Togliatti system
$$I=(x_0^5,x_1^5,x_2^5,x_0^3x_1x_2,x_0x_1^2x_2^2)$$ all $4$-osculating
spaces to $X$  have dimension lower than $14$, which is the expected dimension,
but the dimension of the previous osculating spaces is not constant. Some points of $X$ have $2$-osculating space or $3$-osculating space of dimension less than the general one (they are {\it flexes} of $X$).

This follows from  \cite{P}, where it is proved that the dimension of
the $s$-osculating space at a point  $x\in X$, corresponding to a vertex
$v_x$ of the polytope $P_I$, is maximal if and only if $(P_I\cap \mathbb
Z^2) \setminus A_I$ contains all points out to level $s-1$ with respect
to $v_x$. This means that, after translating $v_x$ to the origin and
using the first lattice points lying along the two edges of $P_I$
emanating from $v_x$ as basis for the lattice, $(P_I\cap \mathbb Z^2)
\setminus A_I$ contains all points $(a,b)$ with $a+b\leq s-1$.  This
remark explains why this example is not included in the list of Perkinson
\cite{P}, Theorem 3.2.

To better understand its geometry, let us note that the surface $X$ is the
projection, from a line $L$, of the blowing up of $\PP^2$ at three general points $E_0$, $E_1$, $E_2$,
embedded in $\PP^{17}$ by the linear system of the quintics through them. The line $L$
is chosen so to meet  all $4$-osculating spaces of this surface. We observe that  there are three lines of this type, obtained by interchanging the variables.
Every such line meets also the $3$-osculating space at one of the three points $E_i$, and the $2$-osculating spaces at the other two. This gives rise to the flexes. Any curve on $X$ corresponding to a general line through one of the blown up points is a smooth rational quartic. One can check that the flexes result to be singular points of intersection of two irreducible components of some reducible quartics obtained after the projection from $L$.
It would be nice to have a precise geometric description of the inflectional loci of $X$, but this goes beyond the scope of this article, we plan to return on this topic in a forthcoming paper.
\end{rem}

\begin{rem} \rm The hypersurface $F_{d-1}$ of degree $d-1$ that contains
the integral points $A_I$ of a minimal
monomial  Togliatti system $$I=(x_0^d,x_1^d,\cdots ,x_n^d,m_1,\cdots
,m_n)\subset k[x_0,\cdots ,x_n]$$  with $\mu(I)=2n+1$ can be described.
It turns out that if $I$ is trivial  then $F_{d-1}$ is the union of
$d-1$ hyperplanes.

If $n=2$, $d=4$ and  $I=(x_0^4, x_1^4, x_2^4, x_0x_1x_2^2, x_0^2x_1^2)$,
then $F_3=(x_0+x_1-3x_2)(3x_0^2-10x_0x_1+3x_1^2-4x_0x_2-4x_1x_2+x_2^2)$.
In this example the surface $X\subset  \PP^{9}$ is the closure of the
image of the parametrization $\phi=\phi_{(I^{-1})_4}$ defined by the
monomials of degree $4$ not in $I$, i.e.
$$(x_0^3x_1, x_0^3x_2, x_0^2x_1x_2, x_0^2x_2^2,  x_0x_1^3, x_0x_1^2x_2,
x_0x_2^3, x_1^3x_2, x_1^2x_2^2,  x_1x_2^3).$$
One computes that its partial derivatives of order $3$ satisfy the
Laplace equation
$$(x_0\phi_{x_0}+x_1\phi_{x_1}-x_2\phi_{x_2})(x_0^2\phi_{x_0^2}-2x_0x_1\phi_{x_0x_1}+x_1^2\phi_{x_1^2}+x_2^2\phi_{x_2^2})=0.$$

If $n=2$, $d=5$ and $I=(x_0^5,x_1^5,x_2^5,x_0^3x_1x_2,x_0x_1^2x_2^2)$
then
$F_4(x_0,x_1,x_2)=24(x_0^4+x_1^4+x_2^4)-154(x_0^3x_1+x_0x_1^3-x_0^3x_2+x_1^3x_2+x_0x_2^3+x_1x_2^3)+269(x_0^2x_1^2+x_0^2x_3^2+x_1^2x_2^2)+
 288(x_0^2x_1x_2+x_0x_1x_2^2)-337x_0x_1^2x_2$ which is irreducible.

Similarly, the Laplace equation satisfied by  the parametrization of the
surface $X\subset \PP^{15}$ is
$$
x_0^4\phi_{x_0^4}+x_1^4\phi_{x_1^4}+x_2^4\phi_{x_2^4}-x_0^3x_1\phi_{x_0^3x_1}-x_0^3x_2\phi_{x_0^3x_2}-x_0x_1^3\phi_{x_0x_1^3}-x_0x_2^3\phi_{x_0x_2^3}-x_1^3x_2\phi_{x_1^3x_2}-x_1x_2^3\phi_{x_1x_2^3}+
$$
$$ + x_0^2x_1^2\phi_{x_0^2x_1^2}+x_0^2x_2^2\phi_{x_0^2x_2^2}+x_1^2x_2^2\phi_
{x_1^2x_2^2}-3x_0^2x_1x_2\phi_{x_0^2x_1x_2}+2x_0x_1^2x_2\phi_
{x_0x_1^2x_2}+2x_0x_1x_2^2\phi_{x_0x_1x_2^2}=0.
$$

\end{rem}

\begin{cor} Fix integers $d\ge 4$ and $n\ge 2$. Let $I=(F_1,\cdots
,F_r)\subset k[x_0,\cdots , x_n]$ be a monomial artinian ideal  of forms
of degree $d$. If  $r\le 2n$ then, for any $s\le d-1$, the
$s$-osculating space to $X$ at a general point $x\in X$  has the
expected dimension, namely ${n+s\choose s}-1$.
\end{cor}

\vskip4mm
In next Theorem we will classify  all smooth minimal  monomial Togliatti
systems $I\in {\mathcal T}^s(n,d)$ whose minimal number of generators
exceed by one the possible minimum. We start with a lemma.

\begin{lem} \label{L_0}
Let $I=(x_0^d,x_1^d,\cdots, x_n^d, m_1, \cdots, m_h)\subset
k[x_0,\cdots, x_n] $ with $h\geq n$, $m_i=x_0^{a_0^i}\cdots
x_n^{a_n^i}$ for $i=1,\cdots,h$,  be a
minimal  Togliatti system of forms of degree $d\ge 3$. Assume
$a_0^1\geq a_0^2\geq \cdots \geq a_0^h$. If
$a_0^{h-n+2}>0$, then $a_0^i>0$ for all index $i$.
\end{lem}
\begin{proof}
Since $I$ is a Togliatti system, there exists a form $F_{d-1}$ of degree
$d-1$  in $x_0, \cdots , x_n$ passing through all points of $A_I$. Its
restriction to 
$H_0$,  $F_{d-1}(0,x_1, \cdots, x_n)$, vanishes at all
points of $A_I^0$. By assumption, to get $A_I^0$ we have to remove from
the simplex $d\Delta_{n-1}$  the $n$ vertices and at most $n-2$ other
points. We denote by $I'\subset K[x_1, \cdots,x_n]$ the ideal generated
by $x_1^d, \cdots, x_n^d$ and the monomials not containing $x_0$ among
$m_1,\cdots,m_h$.
If  $F_{d-1}(0,x_1,
\cdots, x_n)\neq 0$, $I'$ is a Togliatti system in $n$ variables with
$\mu(I')\leq 2n-2$,  which contradicts Theorem  \ref{mainthm1}. Hence
$F_{d-1}(0,x_1, \cdots, x_n)=0$ and $F_{d-1}=L_0F_{d-2}$. But $I$ is
minimal, so by Proposition \ref{failureWLP}  $L_0$ does not contain any
point of $d\Delta_n\setminus A_I$ except the vertices, which implies
that $a_0^i>0$ for any index $i$.
\end{proof}

\begin{rem}\label{trivialB} \rm
Recall that, when $I$ is a monomial Togliatti system, the projective variety $X$ defined by the apolar linear system of forms of degree $d$
has  all $(d-1)$-osculating spaces of  dimension strictly
less than expected, i.e. $X$ satisfies a Laplace equation of order $d-1$. Since
the $(d-1)$-osculating spaces of $V(n,d)$ have the expected dimension, this means that the space that $I$ determines
meets the $(d-1)$-osculating space $\TC_x^{(d-1)}V(n,d)$ for all $x\in V(n,d)$.
As pointed out in \cite{MMO}, \S 4,  when $I$ is as in Lemma \ref{L_0}, i.e. all monomials in $I$
except $x_0^d, \cdots,x_n^d$ are multiple
of one variable,
there is a point $p\in V(n,d)$ such that
the intersection of $I$ with the $(d-1)$-osculating space at  $p$ meets
all the other $(d-1)$-osculating spaces. These Togliatti systems   were called in \cite{MMO} trivial of type B.

For instance, if $t= {n+d-2 \choose
n-1}$  and $F_1, \cdots,F_t$ are any
general monomials of degree $d-1$, the ideal
$$I=(x_0^d,\cdots,x_n^d, x_0(F_1,\cdots,F_t))$$
is a minimal Togliatti system of  the type just described.
\end{rem}

\begin{thm}\label{mainthm2}
 Let $I\subset k[x_0,x_1,\cdots ,x_n]$ be a smooth minimal  monomial
Togliatti system of forms of degree $d\ge 4$.
 Assume that $\mu(I)=2n+2$. Then  $I$ is trivial unless $n=2$ and, up to
a permutation of the coordinates, one of the following cases holds:
 \begin{itemize}
 \item[(i)] $d=5$ and
$I=(x_0^5,x_1^5,x_2^5,x_0^3x_1x_2,x_0^2x_1^2x_2,x_0x_1^3x_2)$ or
$I=(x_0^5,x_1^5,x_2^5,x_0^3x_1x_2,x_0x_1^3x_2,x_0x_1x_2^3)$ or
$I=(x_0^5,x_1^5,x_2^5,x_0^2x_1^2x_2,x_0^2x_1x_2^2,x_0x_1^2x_2^2)$.
  \item[(ii)] $d=7$ and
$I=(x_0^7,x_1^7,x_2^7,x_0^3x_1^3x_2,x_0^3x_1x_2^3,x_0x_1^3x_2^3)$ or
$I=(x_0^7,x_1^7,x_2^7,x_0^5x_1x_2,x_0x_1^5x_2,x_0x_1x_2^5)$ or
$I=(x_0^7,x_1^7,x_2^7,x_0x_1x_2^5,x_0^3x_1^3x_2,x_0^2x_1^2x_2^3)$.
 \end{itemize}
\end{thm}
\begin{proof} Let us first assume that $n=2$ and let
$I=(x_0^d,x_1^d,x_2^d, m_1, m_2, m_3)\subset k[x_0,x_1,x_2] $ with
$m_i=x_0^{a_0^i}x_1^{a_1^i} x_2^{a_2^i}$ and $\sum _{j=0}^2a_j^i =d$ be
a minimal smooth Togliatti system. We distinguish several cases:

\noindent \underline{Case 1.}
We assume that there is $0\le j \le 2$ such that  $a_j^1,a_j^2,a_j^3\ge
2$. Wlog we can assume $j=0$ and $a_0^1\ge a_0^2\ge a_0^3\ge 2$. Let
$F_{d-1}$ be a plane curve containing all  points of $A_I$. Since
$F_{d-1}$ contains the $d$ points of $A_I^1$ and the $d-1$ points of
$A_I^0$ it factorizes as $F_{d-1}=L_0L_1F_{d-3}$.

Let $2\leq i<d$, then $H_i$ contains $d-i+1$ integral points of
$d\Delta_2$; to get $A_I^i$ we have to remove three points, the first
one from $H_{a_0^3}$, the second one from $H_{a_0^2}$, and the third one
from $H_{a_0^1}$.
First of all we want to exclude that  $a_0^3<a_0^2$. Otherwise
$F_{d-3}$ has as factors $L_2, \cdots,L_{a_0^3}$; but in view of
minimality $H_{a_0^3}$ must be contained in $A_I$, which gives a
contradiction. Therefore $a_0^3=a_0^2$ and there are two subcases to
analyze separately:

\begin{itemize}

\item[(1.1)] $a_0^1=a_0^2=a_0^3:=s\ge 2$. In this case $F_{d-1}$
factorizes as  $F_{d-1}=L_0\cdots L_{s-1}L_{s+1}\cdots L_{d-1}$ and the
plane curve $F_{d-1}$ contains all points of $A_I$ if and only if
$s=d-2$. But in this case $m_1=x_0^{d-2}x_1^2$, $m_2=x_0^{d-2}x_1x_2$,
$m_3=x_0^{d-2}x_2^2$, and  applying Proposition \ref{smoothness} we
deduce that $I=(x_0^d,x_1^d,x_2^d, m_1, m_2, m_3)$ is not a smooth
Togliatti system since it violates condition (ii) of Proposition \ref{smoothness}.

 \item[(1.2)] $u:=a_0^1>a_0^2=a_0^3:=s\ge 2$. In this case
$F_{d-1}=L_0L_1F_{d-3}$ and $F_{d-3}$ contains all integral points in
$\cup _{\ell =2}^{d-1} A_I^{\ell }$
if and only if $u=s+1$ and
$(m_1,m_2,m_3)=x_0^sx_1^ax_2^{d-1-a-s}(x_0,x_1,x_2)$  for a suitable
$a\geq 0$. Therefore, $I$ is a trivial smooth Togliatti system.

\end{itemize}

\noindent \underline{Case 2.} We assume $a_j^i\ge 1$ for all $i,j$ and
that for all $0\le j\le 2$  there exists $1\le i_j\le 3$ such that
$a_j^{i_j}= 1$.
We distinguish  4 subcases and a straightforward computation allows us
to conclude:
\begin{itemize}
\item[(2.1)]
$(x_0^d,x_1^d,x_2^d,x_0^{d-2}x_1x_2,x_0x_1^{d-2}x_2,x_0x_1x_2^{d-2})$ is
a smooth minimal Togliatti system if and only if $d=5$ or 7.

\item[(2.2)]
$(x_0^d,x_1^d,x_2^d,x_0^{d-2}x_1x_2,x_0x_1^{d-2}x_2,x_0^ax_1^bx_2^{c})$
with $(a,b,c)\ne (1,1,d-2)$ is a smooth minimal Togliatti system if and
only if $d=5$ and $(a,b,c)=(2,2,1)$.

\item[(2.3)]
$(x_0^d,x_1^d,x_2^d,x_0x_1x_2^{d-2},x_0^ax_1^{b}x_2,x_0^ex_1^fx_2^{g})$
with $a,b\ge 2$ and $(e,f,g)\ne (d-2,1,1),(1,d-2,1),(1,1,d-2)$  is a
smooth minimal Togliatti system if and only if $d=7$, $(a,b)=(3,3)$ and
$(e,f,g)=(2,2,3)$.

 \item[(2.4)]
$(x_0^d,x_1^d,x_2^d,x_0x_1^ax_2^b,x_0^cx_1x_2^e,x_0^fx_1^gx_2)$ is a
smooth minimal Togliatti system if and only if $d=5$ and $a=b=c=e=f=g=2$
or $d=7$ and $a=b=c=e=f=g=3$.
\end{itemize}

\noindent \underline{Case 3.}  We assume that there exists
$a_{j_0}^{i_0}=0$, $a_j^i\ge 1$ for all $(i,j)\ne (i_0,j_0)$  and that
for all $0\le j\le 2$  
 there exists $1\le i_j\le 3$ such that
$a_j^{i_j}= 1$.  The smoothness criterion (Proposition
\ref{smoothness}) implies that, up to permutation of the coordinates, we
have $m_1=x_1^{d-1}x_2$ and we can assume $m_2=x_0^ax_1x_2^b$ and
$m_3=x_0^ux_1^vx_2^w$  with $a,b,u,v,w\ge 1$ and $I$ is never a  smooth
minimal Togliatti system.

\noindent \underline{Case 4.}  We assume that there exists
$a_{j_0}^{i_0}=a_{j_1}^{i_1}=0$,  and that for all $0\le j\le 2$  there
exists $1\le i\le 3$ such that $a_j^{i}\le 1$.  The smoothness criterion
(Proposition \ref{smoothness}) implies that, up to permutation of the
coordinates, we have $m_1=x_1^{d-1}x_2$, $m_2=x_0^{d-1}x_2$ and
$m_3=x_0^ax_1^bx_2^c$ which does not correspond to a smooth minimal
Togliatti system.

\vskip 2mm
 Let us now assume that $n\geq 3$. We want to prove  that  all minimal smooth
monomial Togliatti systems $I\subset k[x_0,\cdots ,x_n]$ of forms of
degree $d\ge 4$ with $\mu(I)=2n+2$ are trivial. This time we distinguish two cases:

\noindent \underline{Case 1.}
For all $0\le j \le n$,  $\# \{i\mid a_j^i\ge 1\}\le 2$. This implies
that each variable $x_j$ appears explicitly in exactly two of the
monomials $m_1, \cdots, m_{n+1}$. Equivalently, looking at the simplex,
the $n+1$ integral points to remove from $d\Delta_n$ to get $A_I$ are
all on the exterior facets, and on each facet there are exactly $n-1$
points. We consider now the restriction of the hypersurface $F_{d-1}$ to
a facet, we apply Theorem \ref{mainthm1} and we get that the
corresponding $n-1$ monomials, together with the $d$th powers of the
corresponding variables, form a trivial Togliatti system in $n$
variables, of the form described in Remark \ref{triv}.  This gives a
contradiction, so this case is impossible.

\noindent \underline{Case 2.} There exists $0\le j \le n$ such that
 $\# \{i \mid a_j^i\ge 1\}\ge 3$. Wlog we can assume $a_0^1\ge a_0^2\ge
\cdots \ge a_0^{n+1} \ge 0$ and $a_0^3\ge 1$. Therefore, in view of
Lemma \ref{L_0}  $a_0^{n+1}>0$.

 This means that all monomials $m_1, \cdots, m_{n+1}$ contain $x_0$. We
consider the restrictions of $x_0^d, \cdots, x_n^d, m_1, \cdots,
m_{n+1}$ to the hyperplane $x_n=x_0+\cdots+x_{n-1}$, they are linearly
dependent by assumption. But in $(x_0+\cdots+x_{n-1})^d$ there is some
monomial not containing $x_0$, that cannot cancel with the others, so
its coefficient  in a null linear combination must be $0$, and by
consequence also the coefficients of $x_1^d, \cdots, x_{n-1}^d$  are
$0$. This implies that the monomials $m_1, \cdots, m_{n+1}$ divided by
$x_0$, together with $x_0^{d-1}, \cdots, x_n^{d-1}$ form again a
Togliatti system but of degree one less, with the same properties. So we can proceed by induction
on the degree, until we arrive to  $d=4$.
Now we have to prove that there is no hypersurface $F_3$ of degree $3$
containing all points of $A_I$ unless $I$ is a trivial monomial
Togliatti system. Since $a_0^{n+1}>0$, $F_3=L_0F_2$ and
$F_2$ contains all points of $A_I\setminus A_I^0$. This is possible  if
and only if $I$ is trivial of type $(x_0^d,\cdots ,x_n^d)+(x_0,\cdots
,x_n)m$ where $m$ is a monomial of degree $d-1$ involving at least $2$
variables.
\end{proof}

\begin{rem}\rm In Theorem \ref{mainthm2}, we did not use the smoothness
assumption in the cases with $n\geq 3$.
\end{rem}

To complete  the results of Theorems \ref{mainthm1} and \ref{mainthm2}, in next Proposition we give a criterion to distinguish the smooth ones among the trivial Togliatti systems. To have a complete picture we also include systems with number of generators bigger than $\rho(n,d)$.

\begin{prop} \label{triv smooth}
Let $I$ be a trivial Togliatti system of the form $(x_0, ... ,x_n)m+(x_0^d,...,x_n^d)$, where $m$ is a monomial. Then $I$ is smooth if and only if one of the following happens (up to permutation of the variables):
\begin{itemize}
\item[(i)] $d=2$ and $n=2$ or $n=3$;
\item[(ii)] $d=3$, $n=2$, $m=x_0^2$;
\item[(iii)] $d\geq 4$, $n=2$, $m=x_0^{d-1}$ or $m=x_0^{i_0}x_1^{i_1}x_2^{i_2}$ with $i_0\geq i_1\geq i_2>0$;
\item[(iv)] $d\geq 4$, $n\geq 3$, $m=x_0^{d-1}$ or $m=x_0^{i_0}x_1^{i_1}\cdots x_n^{i_n}$ with $i_0\geq i_1\geq \cdots \geq i_n\geq 0$ and $i_2>0$.
\end{itemize}
\end{prop}

\begin{proof} If $d=2$, we may assume that $m=x_0$.  If $n=2$  then $X$ is a point. Hence, the system $I$ is smooth. Assume $n\geq 3$. After cutting the points of $I$ from $\Delta$ it remains $A_I=A_I^0$, which is the $(n-1)$-dimensional simplex minus the $n$ vertices.  Through  each vertex of the polytope $P_I$ there are $2(n-2)$ edges. Then  the system is singular unless $n=3$. Indeed by Proposition \ref{smoothness}, (i), for $X$ to be smooth the number of edges emanating from each vertex must be equal to $n-1$.

If $d=3$ then $m$ can be $ x_0^2$, or
$x_0x_1$.
If $n=2$, the first case is smooth, because $P_I$ is a trapezium, and the second one  is singular: indeed, we cut from $\Delta$ the whole edge $x_0^3 - x_1^3$. So an edge of $P_I$ is $x_0^2x_2 - x_1^2x_2$, but the central point $x_0x_1x_2$ does not belong to $A_I$. Therefore this edge gives a singularity.
If $n\geq 3$ both cases are singular: the first one because through the vertices of $P_I$ adjacent to $x_i^3$ there are more than $n$ edges, the second one because  $P_I$ contains the $1$-dimensional faces for $n=2.$

Now assume $d\geq 4$ and $n=2$. If $m=x_0^{d-1}$, then the system is clearly smooth. If $m
=x_0^{d-2}x_1$ then it is singular because the situation is as in  Figure 1.
If $m=x_0^{d-i}x_1^{i-1}$ with $i>2$, the system  is singular because in the edge $x_0^d - x_1^d$ of $P_I$ we have to cut two points in the middle.
Finally if $m=x_0^{i_0}x_1^{i_1}x_2^{i_2}$, with $i_0, i_1, i_2$ all strictly positive,
we get a  smooth system because the points of $I$ are all inner points in $P_I.$

If $d\geq 4$ and $n\geq 3$, then if $m$ is $x_0^{d-1}$, the system is smooth; if $m=x_0^{d-2}x_1$ or $m=x_0^{d-i}x_1^{i-1}$ with $i>2$   the system is singular, because $P_I$ has a $2$-dimensional face which is singular. Finally
if $m$ contains at least $3$ of the variables the system is smooth: indeed on the $1$-dimensional edges of $P_I$ there are no points of $I$, while on the faces of $ P_I$ of dimension at least $2$ the points of $I$ are in the interior.
\end{proof}

\section{Number of generators of a minimal Togliatti system}\label{number}

We consider now the range comprised between $\mu^s(n,d)$ and $\rho^s(n,d)$ (resp.
$\mu(n,d)$ and $\rho(n,d)$)  and ask if all values are reached.

Next Proposition gives a rather precise picture in the case $n=2$.
\begin{prop} \label{interval}  With  notation as in Section \ref{minimalnumbergenerators} we have:
\begin{itemize}
\item[(i)] For any $d\ge 4$, $\mu ^s(2,d)=\mu (2,d)=5$.
\item[(ii)] For any $d\ge 4$, $\rho ^s(2,d)=\rho (2,d)=d+1$.
\item[(iii)] For any $d\ge 4$ and any $5\le r \le d+1$, there exists
$I\in \mathcal{ T}^s(2,d)$ with $\mu (I)=r$.
\end{itemize}
\end{prop}
\begin{proof} (i) It follows from Theorem \ref{mainthm1}.

(ii) By definition we have $\rho (2,d)\le d+1$ for any $d\ge 4$. The
inequality $\rho ^s(2,d)\ge d+1$ (and, hence,  $\rho ^s(2,d)=\rho
(2,d)=d+1$) will follows from (iii).

(iii) For any $d\ge 4$ and for any $5\le r\le d+1$, we consider the ideals
$$I_5=(x_0^d,x_1^d,x_2^d)+x_0^{d-1}(x_1,x_2), \text{ and for } r>5$$
$$I_r=(x_0^d,x_1^d,x_2^d)+x_0^{d-r+3}x_1x_2(x_0^{r-5},x_0^{r-6}x_1,\cdots ,x_0x_1^{r-6},x_1^{r-5},x_2^{r-5}).$$
We have $\mu(I_r)=r$ and it follows from Propositions \ref{failureWLP}
and \ref{smoothness} that $I_r\in \mathcal{ T}^s(2,d)\subset
k[x_0,x_1,x_2]$,  which proves what we want.
\end{proof}

\begin{rem} \rm
 Proposition \ref{interval} does not generalize to the case $n\ge 3$,
i.e. not all values of $r $, $\mu^s(n,d)  \le r \le \rho^s(n,d)$, occur
as the minimal number of generators of a smooth Togliatti system $I\in
\mathcal{ T}^s(n,d)$.
The first case is illustrated in next Lemma for the case $d=3$ and
next  Proposition for the general  case $d\ge 4$.
\end{rem}

\begin{lem}\label{d=3,2n+3}
Assume $n\ge 4$ and let $I$ be a minimal Togliatti system of cubics.
Then, $\mu (I)\ge 2n+1$. In addition, we have:

\begin{itemize}
\item[(i)] $\mu(I)=2n+1$ if and only if  $I$ is trivial, i.e., up to
permutations of the coordinates, $I=(x_0^3, \cdots
,x_n^3)+x_0^2(x_1,\cdots ,x_n)$. In particular, $I\in {\mathcal
T}(n,3)\setminus {\mathcal T}^s(n,3)$.
\item[(ii)] $\mu(I)=2n+2$ if and only if  $I$ is trivial, i.e., up to
permutations of the coordinates, $I=(x_0^3, \cdots
,x_n^3)+x_ix_j(x_0,\cdots ,x_n)$ with $i\ne j$. In particular, $I\in
{\mathcal T}(n,3)\setminus {\mathcal T}^s(n,3)$.
\item[(iii)] $\mu(I)\ne 2n+3$.
\end{itemize}
\end{lem}
\begin{proof} We proceed by induction on $n$. With Macaulay2 (\cite{M}) we easily
check that $\mu(I)\ge 9$ for any $I\in {\mathcal T}(4,3)$. Assume now
$n\ge 5$ and suppose that the result is true for $n-1$. We take
$I=(x_0^3, \cdots ,x_n^3,m_1,\cdots ,m_{n-1})$ with
$m_i=x_0^{a_0^i}\cdots x_n^{a_n^i}$, $a_0^i+\cdots +a_n^i=3$, and we
will see that there is no  hyperquadrics $F_2$ containing all points of
$A_I$. Assume it exists and we will get a contradiction. Wlog we can
assume  that $x_0$ appears explicitly in the monomial $m_1$ and $A_I^0$
is equal to $3\Delta _{n-1}$ minus $n$ vertices and at most $n-2$ other
points. By induction no hyperquadric in $x_1, \cdots ,x_n$ contains
$A_I^0$. So $F_2$ decomposes as $F_2=L_0F_1$ and since there is no
hyperplane $F_1$ containing all the points of $A_I\setminus A_I^0$
we get a contradiction.

Let us now classify all Togliatti systems $I\in {\mathcal T}(n,3)$,
$n\ge 4$, with $2n+1\le \mu(I)\le 2n+3$.

(i)  Assume $n=4$, $I\in {\mathcal T}(4,3)$ and $ \mu(I)=2n+1$. Using
Macaulay2 we get that $I$ is trivial. Suppose now $n\ge 5$, let
$I=(x_0^3,\cdots ,x_n^3, m_1,\cdots, m_n)\in {\mathcal T}(n,3)$ with
$m_i=x_0^{a_0^i}x_1^{a_1^i}\cdots x_n^{a_n^i}$ and $\sum _{j=0}^na_j^i
=3$, and let $F_2$ a hyperquadric passing through the points of $A_I$.
Wlog we can assume $a_0^1,a_0^2\ge 1$. Therefore, $F_2$ factorizes as
$F_{2}=L_0L_1$ and since $F_{2}$ cannot miss any point of $A_I$  we must
have $A_I^{2}=\emptyset $ which  forces $m_1=x_0^{2}x_1, \cdots ,
m_n=x_0^{2}x_{n}$ and hence $I$ is trivial.

(ii) Using Macaulay2 we prove that if $n=4$, $I\in {\mathcal T}(4,3)$
and $ \mu(I)=2n+2$ then $I$ is trivial. Suppose now $n\ge 5$ and let
$I=(x_0^3,\cdots ,x_n^3, m_1,\cdots, m_{n+1})$ with
$m_i=x_0^{a_0^i}x_1^{a_1^i}\cdots x_n^{a_n^i}$ and $\sum _{j=0}^na_j^i
=3$. Wlog we can assume $a_0^1\ge \cdots \ge a_0^{n+1}\ge 0$ and
$a_0^1>0$. If $a_0^3>0$ then $a_0^{n+1}>0$ by Lemma \ref{L_0} and
$F_2=L_0F_1$ where $F_1$ is a hyperplane containing all points of
$A_I\setminus A_I^0$.  This is possible if and only if $I$ is trivial of
type $I=(x_0^3, \cdots ,x_n^3)+x_ix_j(x_1,\cdots ,x_n)$ with $i\ne j$.
If $a_0^3=0$ then using hypothesis of induction together with the
fact that $a_0^1>0$ we get that the restriction of $x_0^3,\cdots ,x_n^3,
m_1,\cdots, m_{n+1}$ to the hyperplane $x_0=0$ is trivial of type
$(x_1^3, \cdots ,x_n^3)+x_1^2(x_2,\cdots ,x_n)$ or
$(x_1^3, \cdots ,x_n^3)+x_ix_j(x_1,\cdots ,x_n)$ with $1\le i< j\le n$.
Therefore, either $I=(x_0^3,x_1^3, \cdots ,x_n^3)+x_1^2(x_2,\cdots ,x_n)$ or
$I=(x_0^3,x_1^3, \cdots ,x_n^3)+x_ix_j(x_1,\cdots ,x_n)$ with $1\le i<
j\le n$; and none of them belongs to ${\mathcal T}(n,3)$.

(iii) Again using Macaulay 2 we prove that the result is true for
$n=4$.  Suppose now $n\ge 5$ and let $I=(x_0^3,\cdots ,x_n^3,
m_1,\cdots, m_{n+2})$ with $m_i=x_0^{a_0^i}x_1^{a_1^i}\cdots
x_n^{a_n^i}$ and $\sum _{j=0}^na_j^i =3$. Wlog we can assume $a_0^1\ge
\cdots \ge a_0^{n+2}\ge 0$ and $a_0^1>0$. If $a_0^4>0$ then
$a_0^{n+1}>0$ by Lemma \ref{L_0} and $F_2=L_0F_1$ but this is impossible
since there is no a hyperplane containing all points of $A_I\setminus
A_I^0$ and no point of $3\Delta _n\setminus A_I$ a part from the
vertices. If $a_0^4=0$ then using hypothesis of induction together with the
fact that $a_0^1>0$ we get that the restriction of $x_0^3,\cdots ,x_n^3,
m_1,\cdots, m_{n+2}$ to the hyperplane $x_0=0$ is trivial of type
$(x_1^3, \cdots ,x_n^3)+x_1^2(x_2,\cdots ,x_n)$ or
$(x_1^3, \cdots ,x_n^3)+x_ix_j(x_1,\cdots ,x_n)$, $1\le i< j\le n$, or
$(x_1^3, \cdots ,x_n^3)+x_1^2(x_2,\cdots ,x_n)+(x_{i_1}x_{i_2}x_{i_3})$,
$1\le i_1\le \i_2\le i_3\le n$ or
$(x_1^3, \cdots ,x_n^3)+x_ix_j(x_1,\cdots
,x_n)+(x_{i_1}x_{i_2}x_{i_3})$, $1\le i< j\le n$, $1\le i_1\le \i_2\le
i_3\le n$. Therefore, $I=(x_0^3,x_1^3, \cdots ,x_n^3)+x_1^2(x_2,\cdots
,x_n)$ or $I=(x_0^3,x_1^3, \cdots ,x_n^3)+x_ix_j(x_1,\cdots ,x_n)$,
$1\le i< j\le n$ or
  $I=(x_0^3,x_1^3, \cdots ,x_n^3)+x_1^2(x_2,\cdots
,x_n)+(x_{i_1}x_{i_2}x_{i_3})$, $1\le i_1\le \i_2\le i_3\le n$ or
$I=(x_0^3,x_1^3, \cdots ,x_n^3)+x_ix_j(x_1,\cdots
,x_n)+(x_{i_1}x_{i_2}x_{i_3})$, $1\le i< j\le n$, $1\le i_1\le \i_2\le
i_3\le n$; and none of them belongs to ${\mathcal T}(n,3)$.
\end{proof}

\begin{prop}\label{2n+3}
Let $n\geq 3$ and $d\geq 4$. Then
there is no $I \in \mathcal{ T}^s(n,d)$ with  $\mu(I)=2n+3$.
\end{prop}

\begin{proof}
We distinguish two cases:
\begin{enumerate}
\item For all $0\le j \le n$,  $\# \{i\mid a_j^i\ge 1\}\le 3$, i.e.
every variable appears in at most three of the monomials $m_1,
\cdots,m_{n+2}$.

If one of the monomials contains all the variables, the other $n+1$
monomials contain two variables each, and we are in the same situation
of Theorem \ref{mainthm2}, Case 1, which is impossible. Therefore no
monomial contains all variables and at least two variables appear in
three monomials. Assume that $x_0$ appears in three monomials; then
$F_{d-1}$ passes through the integral points of $A_I^0$. Recall that
$A_I^0$  is equal to $d\Delta_{n-1}$ minus the $n$ vertices and $n-1$
other points. So the removed points form a Togliatti system $I'$ in the
$n$ variables $x_1, \cdots,x_n$ with $\mu=2n-1$ and we can apply Theorem
\ref{mainthm1}. There are two possibilities:
\begin{itemize}
\item[(i)] $n=3$ and $I'$ is one of the two special Togliatti systems
of degree $5$ or  $4$ of Theorem \ref{mainthm1}. If  $d=5$, up to
permutation of the variables the only possibility is $I=(x_0^5, \cdots ,
x_3^5,x_0^4x_2, x_0^4x_3, x_1^3x_2x_3,x_1x_2^2x_3^2,x_0^ax_1^b)$ with
$a,b>0$. But it is easy to check that this is not a Togliatti system. In
the case $d=4$ there are two possibilities: $I=(x_0^4, \cdots ,
x_3^4,x_0^3x_2, x_0^3x_3, x_1x_2x_3^2,x_1^2x_2^2,x_0^ax_1^bx_3^c)$ with
$a,b>0$, $c\geq 0$, or $(x_0^4, \cdots , x_3^4,x_0^2x_1x_2, x_1x_2x_3^2,
x_1^2x_2^2,x_0^ax_3^b, x_0^cx_3^d)$ with $a,b,c,d>0$. Both systems are
not Togliatti.
\item[(ii)] $I'$ is of the form $(x_1^d, \cdots,x_n^d)+x_1^{d-1}(x_2,
\cdots, x_n)$.  In this case  $x_1$ appears in at least $n-1$ monomials,
therefore $n=3$ or $n=4$.

If $n=3$,  the other three monomials in $I$ are either of the form
$x_0^{d-1}(x_2, x_3), x_0^ax_1^bx_2^c$, or of the form $x_0^{d-1}(x_1,
x_3), x_0^ax_2^bx_3^c$,  with $a>0$, $b>0$, $c\geq 0$. It is immediate
to check that they are not Togliatti systems.
 If $n=4$ then the six monomials $m_1, \cdots, m_6$ are of the form
$x_0^{d-1}(x_2,x_3,x_4), x_1^{d-1}(x_2,x_3,x_4)$. Also in this case the
system is not Togliatti.
\end{itemize}
\item There exists an index $j$ such that $\# \{i\mid a_j^i\ge 1\}\geq
4$, i.e. one of the variables appears in at least $4$ monomials. We can
assume $j=0$. Therefore, by Lemma \ref{L_0}, $x_0$ appears in all
monomials $m_1, \cdots, m_{n+2}$.  Let $m_i'=m_i/x_0$, $i=1, \cdots,
n+2$. As in the proof of Theorem \ref{mainthm2}, case 2, we observe that
$m_1', \cdots,m'_{n+2}$, together with $x_0^{d-1}, \cdots, x_n^{d-1}$,
form a Togliatti system $I_1$ of degree $d-1$. We
distinguish the following possibilities:
\begin{itemize}
\item[(i)] at least one of the monomials $m'_i$ is the $(d-1)$th power
of a variable, so $\mu(I_1)<2n+3$; or
\item[(ii)] $\mu(I_1)=\mu(I)=2n+3$.
\end{itemize}
In case (i), if $d>4$, $I_1$ is trivial, which implies that $I$ contains
a trivial Togliatti system and therefore is non minimal: contradiction.
If $d=4$, $I_1\in \mathcal{ T}(n,3)$ and $\mu(I_1)\leq 2n+2$.

In case (ii), we can apply  the  above argument to $I_1$, and so on, by induction.

In
any case, applying repeatedly this procedure, possibly involving
different variables, we arrive to a Togliatti system $I_1$ of degree $d=3$
with $\mu\leq 2n+3$, which is obtained from $I$  dividing the monomials
$m_1, \cdots, m_{n+2}$ by a common monomial factor $M$.  If $n=3$,
we conclude with the help of Macaulay2.
If $n\geq 4$,  by Lemma \ref{d=3,2n+3}, $I_1$ is trivial of
type $(x_0^3,\cdots,x_n^3)+x_0^2(x_1,\cdots,x_n)$ or
$(x_0^3,\cdots,x_n^3)+x_ix_j(x_0,\cdots,x_n)$.  In both cases  $I$ is
 not minimal and we are done.
\end{enumerate}
\end{proof}

\begin{rem} \rm If $n=3$ and $d=4$ one can check with the help of
Macaulay2 that there exist two
types of minimal Togliatti systems $I$ with $\mu(I)=2n+3=9$, both non
smooth,  precisely $(x_0^4, x_1^4, x_2^4, x_3^4)+x_0^2(x_0x_2,x_0x_3,
x_1^2,x_1x_2,x_1x_3)$ and $(x_0^4, x_1^4, x_2^4, x_3^4)+x_0^2(x_1^2,x_1x_2,x_2^2,x_0x_3,
x_3^2)$.
\end{rem}

We note that if $d=2$ the ideal $I=(x_0, x_1)^2+(x_2,x_3,x_4,x_5)^2$,
with $\mu(I)=2n+3=13$, belongs to $\mathcal{ T}^s(5,2)$, while if $d=3$
then $2n+3<\mu^s(n,3)$ for any $n\geq 4$.

Computations made with Macaulay2 illustrate the  complexity of the
general case. However, some ranges and some sporadic values can be
covered. For example:
\begin{ex} \rm
For any $d>n\ge 3$ and for any $r$, ${d+n-2\choose n-2}+n+2\le r\le
{d+n-2\choose n-2}+d+1$, there exists $I\in \mathcal{ T}^s(n,d)$ with
$\mu(I)=r$. (Notice that when $n=3$ we have
$d+6\le r\le 2d+2$). In fact, it is enough to take $$I=(x_0,x_1,\cdots,
x_{n-2})^d+(x_{n-1}^d,x_n^d)+(x_{n-1},x_n)^{d-h}m'$$
where $2\le h\le d-n+1$ and $m'$ is a monomial of
degree $h$ containing only $x_0, \cdots ,x_{n-2}$.
\end{ex}

Nevertheless if we delete the smoothness hypothesis, we can generalize
Proposition \ref{interval} and we get
\begin{prop} \label{interval2}  With the above notation we have:
\begin{itemize}
\item[(i)] For any $d\ge 4$, $\mu (n,d)=2n+1$.
\item[(ii)] For any $d\ge 4$, $\rho (n,d)={n+d-1\choose n-1}$.
\item[(iii)] For any $d\ge 4$, $n=3$ and any integer $r$ with $\mu (3,d)=7\le r \le \rho
(3,d)= {d+2\choose
2}$, there exists $I\in \mathcal{ T}(3,d)$ with $\mu (I)=r$.
\end{itemize}
\end{prop}

\begin{proof} (i) It follows from Theorem \ref{mainthm1}.

(ii) By definition we have $\rho (n,d)\le {n+d-1\choose n-1}$ for any
$d\ge 4$. Let us prove that  $\rho (n,d)\ge {n+d-1\choose n-1}$, i.e.
there exists $I\in  \mathcal{ T}(n,d)$ with
$\mu (I)={n+d-1\choose n-1}$. Consider
$$I=(x_0^d,x_1^d,...,x_n^d)+x_1(x_1,...,x_n)^{d-1}+x_2(x_2,...,x_n)^{d-1}+...+x_{n-2}(x_{n-2},x_{n-1},x_n)^{d-1}+x_0^3(x_{n-1},x_n)^{d-3}.$$
We have $$ \begin{array}{ccl} \mu(I) & = & n+1+\sum
_{i=2}^{n-1}[{d-1+i\choose i}-1] + d-2 \\
& = &
d+1+\sum _{i=2}^{n-1}{d-1+i\choose i} \\ &  = & \sum
_{i=0}^{n-1}{d-1+i\choose i}\\ &  = & {d-1+n\choose n-1}.\end{array}$$

When we substitute $x_0$ by $x_1+x_2+...+x_n$ the $\mu(I)$ generators of
$I$ become $k$-linearly dependent; so $I$ fails WLP in degree $d-1$
(Theorem \ref{teathm}) and $I$ is minimal because no proper subset of
the generators of $I$ defines a Togliatti system. Therefore, $I\in
\mathcal{ T}(n,d)$.

(iii) Assume $n=3$. For $r=7$
we take $I=(x_0^d,x_1^d,x_2^d,x_3^d)+x_0^{d-1}(x_1,x_2,x_3)$, for $r=8$
we take $ I=(x_0^d,x_1^d,x_2^d,x_3^d)+x_0^{d-2}x_1(x_0,x_1,x_2,x_3)$ and
for $r=9$ we take
$I=(x_0^d,x_1^d,x_2^d,x_3^d)+x_0^{d-2}(x_1^2,x_0x_1,x_2^2,x_2x_3,x_3^2)$.

We
will now proceed by induction on $d$. In the case $d=4$ we exhibit an explicit example for any $10 \le r\le 14$ (note that the case $r=15$ is covered by the example given in (ii)):
\begin{itemize}
\item  $r=10$:  $(x_0,x_1)^4+ (x_2,x_3)^4$ (smooth);
\item $r=11$: $(x_0,x_1)^4+(x_2^4,x_2^3x_3,x_2^2x_3^2,x_3^4,x_0x_2x_3^2,x_1x_2x_3^2)$;
\item $r=12$: $(x_0,x_1)^4+(x_2^4,x_2^3x_3,x_2x_3^3,x_3^4,x_0^2x_3^2,x_0x_1x_3^2,x_1^2x_3^2)$;
\item $r=13$: $(x_0,x_1)^4+(x_2^4,x_2^3x_3,x_2x_3^3,x_3^4,x_0^3x_3,x_0^2x_1x_3,x_0x_1^2x_3,x_1^3x_3)$;
\item if $r=14$: the systems described in  Remark \ref{trivialB} work in this case.
\end{itemize}
We suppose now $d>4$ and we will prove that  for any
$7\le r\le {d+2\choose 2}$ there exists $I\in {\mathcal T}(3,d)$ with
$\mu (I)=r$.

Indeed, for any $7\le s\le {d+1\choose 2}$ we take $J\in {\mathcal
T}(3,d-1)$ with $\mu (J)=s$  and we define $I=(x_0^d,x_1^d,x_2^d)+x_3J$.
Note that $I\in  {\mathcal T}(3,d)$ and $10\le \mu(I)=\mu(J)+3\le
{d+1\choose 2}+3$.
 Observe also that $I=(x_0^d,x_1^d,x_2^d,x_3^d)+x_0(x_1,x_2,x_3)^{d-1}
\in  {\mathcal T}(3,d)$ and $\mu(I)={d+1\choose 2}+4$. So, it only
remains to cover the values of $r$, ${d+1\choose 2}+4<r\le {d+2\choose
2}.$ To this end,  for any $3\leq i \leq d-1$ we define
$$I_{i}=(x_0^d,x_1^d,x_2^d,x_3^d)+(x_1^{i_1}x_2^{i_2}x_3^{i_3} \mid
i_1+i_2+i_3=d, \ 1\le i_1<d)+x_0^{i}(x_2,x_3)^{d-i}.$$
First of all we observe that $\mu (I_i)={d+2\choose 2}+3-i$. Therefore,
when $i $ ranges from $i=3$ to $d-1$ we sweep the interval $[{d+1\choose
2}+5, {d+2\choose 2}]$. By Proposition \ref{failureWLP} to prove that
$I_i\in  {\mathcal T}(3,d)$
it is enough to show that there is a surface $F_{d-1}$ of degree $d-1$
containing all integral points of $A_{I_i}$. Since
$A_{I_i}^1=(d-1)\Delta _2, \cdots , A_{I_i}^{i-1}=(d-i+1)\Delta _2$, we
have $F_{d-1}=L_1\cdots L_{i-1}F_{d-i}$ where $F_{d-i}$ is a surface of
degree $d-i$ containing all integral points of $A_{I_i}\setminus
\cup_{j=1}^{i-1}A_{I_i}^j$. The surfaces $F_{d-i}$ of degree $d-i$ are
parametrized by a $k$-vector space of dimension ${d-i+3\choose 3}$. On
the other hand, to contain the aligned $d-1$ points of $A_{I_i}^0$
imposes $d-i+1$ conditions on the surfaces of degree $d-i$, to contain
the points of $A_{I_i}^{i+1}= (d-i-1)\Delta _2, \cdots
,A_{I_i}^{d-1}=\Delta _2$ imposes ${d-i+1\choose 2}, \cdots , 3$
conditions, respectively, and finally  to contain the points of
$A_{I_i}^{i}$ imposes ${d-i+2\choose 2}-(d-i+1)$ conditions.  Summing up
we have ${d-i+3\choose 3}-1$ conditions. Therefore, there exists at
least a surface $F_{d-i}$ of degree $d-i$ through all integral
  points of $A_{I_i}\setminus \cup_{j=1}^{i-1}A_{I_i}^j$ and, hence  a
surface $F_{d-1}=L_1\cdots L_{i-1}F_{d-i}$ of degree $d-1$ containing
all integral points of $A_{I_i}$.
\end{proof}

\begin{rem} \rm
For $n=3, d=4$, with Macaulay2 we have obtained the list of all
minimal Togliatti systems with $\mu(I)\leq 13$. The computations become too
heavy for $\mu=14, 15$.

\end{rem}


\section{On the stability of the associated syzygy bundles}
\label{associatedbundles}

In this section we restrict our attention to the case $n=2$ and we will
analyze whether the syzygy bundle $E_I$ on $\PP^2$ associated to a
minimal smooth monomial Togliatti system $I\in  \mathcal{ T}(2,d)$  is
$\mu $-(semi)stable.

\begin{defn} \rm  A syzygy bundle $E_{d_1,\cdots ,d_r}$ on $\PP^n$ is a
rank $r-1$ vector bundle defined as the kernel
of an epimorphism
$$(f_1, \cdots , f_r) : \oplus _{i=1}^r \mathcal{
O}_{\PP^n}(-d_i)\longrightarrow \mathcal{ O}_{\PP^n}$$
where $(f_1,\cdots , f_r)\subset k[x_0,x_1\cdots ,x_n]$ is an artinian
ideal, and $d_i = deg(f_i)$.
When $d_1 = d_2 = \cdots  = d_r = d$, we write $E_{d,n}$ instead of
$E_{d_1,\cdots ,d_r}$.
\end{defn}

\begin{defn} \rm
Let E be a vector bundle  on $\PP^n$ and set
$$\mu(E) :=
\frac{c_1(E)}{rk(E)}.$$
The vector bundle $E$ is said to be $\mu$-semistable in the sense of
Mumford-Takemoto if
$\mu(F)\le \mu(E)$
for all non-zero subsheaves $F \subset E$ with $rk(F) < rk(E)$; if
strict inequality holds then $E$
is $\mu $-stable.
\end{defn}
Note that for  a rank $s$ vector bundle $E$ on $\PP^n$, with $(c_1(E), s)
= 1$, the concepts
of $\mu$-stability and $\mu$-semistability coincide.

\vskip 2mm Using Klyachko results on toric bundles (\cite{K},
\cite{K1} and \cite{K2}), Brenner deduced the following nice
combinatoric criterion for the (semi)stability of  the syzygy bundle
$E_{d_1,\ldots,d_r}$ in the case where the associated forms
$f_1,\ldots,f_r$ are all monomials. Indeed, we have

\begin{prop}\label{combinatoria} Let $I=(m_1,\cdots ,m_r)\subset k[x_0,
\cdots ,x_n]$ be a  monomial artinian ideal. Set  $d_i=\deg (m_i)$. Then
the syzygy bundle $E_{d_1,\ldots,d_r}$ on $\PP^n$ associated to $I$, is
$\mu$-semistable (resp. $\mu$-stable) if and only if for every
$J=(m_{i_1},\cdots ,m_{i_s}) \varsubsetneq I$, $s\ge 2$, the inequality
\begin{equation} \label{ineq1}
\frac{d_J-\sum _{j=1}^s d_{j_i}}{s-1}\le
    \frac{-\sum _{i=1}^rd_i}{r-1}\qquad\mbox{(resp.\ $<$)}
\end{equation}
holds, where $d_J$ is the degree of the greatest common factor of the monomials
$m_{j_i}\in J$.
\end{prop}
\begin{proof}
See \cite{B}, Proposition 2.2 and Corollary 6.4.
\end{proof}

\begin{ex} \rm
(1) If we consider the monomial artinian ideal
$I:=(x_0^5,x_1^5,x_2^5,x_0^{2}x_1^{2}x_2)\subset k[x_0,x_1,x_2]$,
inequality (\ref{ineq1}) is strictly fulfilled for any proper subset
$J\varsubsetneq \{x_0^5,x_1^5,x_2^5,x_0^{2}x_1^{2}x_2\}$. Therefore the syzygy bundle $E$ associated to $I$ is
$\mu $-stable.

(2) If we consider the monomial artinian ideal
$I:=(x_0^5,x_1^5,x_2^5,x_0^{4}x_1)\subset k[x_0,x_1,x_2]$, then for the
subset
$J:=\big\{x_0^5,x_0^{4}x_1 \big\}$  inequality (\ref{ineq1}) is not
fulfilled. Therefore the syzygy bundle $E_I$ associated to $I$ is not
$\mu$-stable. In fact, the slope of $E_I$ is
$\mu(E_I)=-20/3$ and the syzygy sheaf $F$ associated to $J$ is a
subsheaf of $E_I$ with slope $\mu(F)=-6$. Since $\mu(F)\nleq
\mu(E_I)$, we conclude that $E$ is not $\mu$-stable.
\end{ex}

\begin{rem} \label{rem} \rm
 Let $I$ be a monomial  artinian ideal generated by $r$ monomials $m_1, \cdots , m_r$ of
degree $d$. It easily follows from the above proposition that the syzygy
bundle $E_{d,n}$ on $\PP^n$ associated to $I$ is $\mu$-(semi)stable if
and only if for every subset $J=\{m_{i_1},\cdots ,m_{i_s}\}\varsubsetneq \{m_1, \cdots , m_r\}$ with $s:=|J|\ge 2$,
\begin{equation} \label{ineq3}
(d-d_J)r+d_J-sd >0 \quad (\mbox{resp.\ } \ge 0),
\end{equation}
where $d_J$ is the degree of the greatest common factor of the monomials
in $J$.
\end{rem}

\begin{thm}
 Let $I\subset k[x_0,x_1,x_2]$ be a smooth minimal  monomial Togliatti
system of forms of degree $d\ge 4$.
 Assume that $\mu(I)\le 6$. Let $E_I$ be the syzygy bundle  associated
to $I$. We have:
 \begin{itemize}
 \item[(a)]  $E_I$ is $\mu$-stable if and only if, up to a permutation
of the coordinates,   one of the following cases holds:
 \begin{itemize}
 \item[(i)]  $\mu(I)=5$, $d=5$ and
$I_1=(x_0^5,x_1^5,x_2^5,x_0^3x_1x_2,x_0x_1^2x_2^2)$.
\item[(ii)] $\mu (I)=6$, $d=7$ and
$I_2=(x_0^7,x_1^7,x_2^7,x_0^3x_1^3x_2,x_0^3x_1x_2^3,x_0x_1^3x_2^3)$ or
$I_3=(x_0^7,x_1^7,x_2^7,x_0^5x_1x_2,$ $x_0x_1^5x_2,x_0x_1x_2^5)$ or
$I_4=(x_0^7,x_1^7,x_2^7,x_0x_1x_2^5,x_0^3x_1^3x_2,x_0^2x_1^2x_2^3)$.
 \end{itemize}

\item[ (b)] $E_I$ is properly $\mu$-semistable if and only if, up to a
permutation of the coordinates,  one of the following cases holds:
 \begin{itemize}

 \item[(i)] $\mu(I)=6$, $d=5$ and
$I_5=(x_0^5,x_1^5,x_2^5,x_0^3x_1x_2,x_0x_1^3x_2,x_0x_1x_2^3)$.
\item[(ii)]
 $\mu(I)=6$, $d=5$ and
$I_6=(x_0^5,x_1^5,x_2^5,x_0^3x_1x_2,x_0^2x_1^2x_2,x_0x_1^3x_2)$ or
$I_7=(x_0^5,x_1^5,x_2^5,x_0^2x_1^2x_2,$ $x_0^2x_1x_2^2,x_0x_1^2x_2^2)$.
 \end{itemize}
\item[ (c)] In all other cases, $E_I$ is unstable.
\end{itemize}
\end{thm}

\begin{proof}  First of all, by Theorem \ref{mainthm1}, we have
$\mu(I)=$5 or 6. Using the classification of Togliatti systems $I \in
\mathcal{ T}(2,d)$ with $5\le \mu(I)\le 6$ given in Theorems
\ref{mainthm1} and \ref{mainthm2}, it is enough to check:

(1) $I_i$, $1\le i\le 4$ corresponds to $\mu$-stable bundles.

(2) $I_i$, $5\le i\le 7$ corresponds to properly $\mu$-semistable bundles.

(3) Trivial Togliatti systems $I \in \mathcal{ T}(2,d)$ correspond to
$\mu$-unstable bundles.

To prove (1) it is enough to observe that inequality (\ref{ineq3}) is
strictly fulfilled for any proper subset $J_i\varsubsetneq I_i$, $1\le
i\le 4$,  with $|J_i|\ge 2$.

To prove (2) we check that inequality (\ref{ineq3}) is satisfied for any
proper subset $J_i\varsubsetneq I_i$, $5\le i\le 7$,  with $|J_i|\ge 2$
and there is
a subset $J_i^0\varsubsetneq I_i$, $5\le i\le 7$,  with $|J_i^0|\ge 2$
and verifying $(d-d_{J_i^0})\mu(I_i)+d_{J_i^0} -d\mu(J_i^0)=0$. For
instance, for $
I_6=(x_0^5,x_1^5,x_2^5,x_0^3x_1x_2,x_0^2x_1^2x_2,x_0x_1^3x_2)$ it is
enough to take $ J_6^0=(x_0^3x_1x_2,x_0^2x_1^2x_2)\subset I_6$ since
$(d-d_{J_6^0})\mu(I_6)+d_{J_6^0} -d\mu(J_6^0)=(5-4)6+4-2\times 5=0$.

(3) Finally let us check that the syzygy bundle $E_I$ associated to
trivial Togliatti systems $I=(x_0,x_1,x_2)m+(m_1,\cdots ,m_{r-3})\in
\mathcal{ T}(2,d)$ are always $\mu$-unstable. Note that $m$ is a
monomial of degree $d-1$ and $m_i$, $1\le i \le r-3$, are monomials of
degree $d$.  For the subset $J=(x_0m,x_1m,x_2m)\subset I$ inequality
(\ref{ineq3}) becomes $(d-(d-1))r+(d-1)-3d>0$ and $E_I$ is
$\mu$-unstable. Indeed, the slope of $E_I$ is $\mu(E_I)=\frac{dr}{r-1}$
and the syzygy sheaf $F$ associated to $J$ is a subsheaf of $E_I$ with
slope $\mu (F)=\frac{3(d-1)}{2}$. Therefore, $\mu(F)\nleq \mu(E_I)$  and
we conclude that $E_I$ is $\mu$-unstable.

\end{proof}

\end{document}